\documentclass[psamsfonts,11pt,reqno]{amsart}

\usepackage{amsmath,amssymb,amsthm,enumerate}
\usepackage[latin1]{inputenc}
\usepackage[english]{babel}

\usepackage{algorithm,algpseudocode}
\usepackage{graphicx}

\usepackage[bookmarks,colorlinks, citecolor = blue]{hyperref}
\usepackage[x11names,rgb]{xcolor}
\usepackage{tikz}
\usetikzlibrary{arrows,shapes}

\numberwithin{equation}{section}
\newtheorem{theorem}{Theorem}[section]
\newtheorem{corollary}[theorem]{Corollary}
\newtheorem{proposition}[theorem]{Proposition}
\newtheorem{lemma}[theorem]{Lemma}

\theoremstyle{definition}
\newtheorem{definition}[theorem]{Definition}
\newtheorem{example}[theorem]{Example}
\newtheorem{remark}[theorem]{Remark}

\def\gm{{\mathfrak m}}

\def\cn{{\mathcal N}}

\def\mathbi#1{\textbf{\em #1}}

\newcommand{\hilbp}{{\mathcal{H}\textnormal{ilb}_{\mathbi{p}(z)}^n}}
\newcommand{\hilbd}{{\mathcal{H}\textnormal{ilb}_d^n}}
\newcommand{\hilb}{{\mathcal{H}\textnormal{ilb}}}

\newcommand{\sat}{\textnormal{sat}}
\newcommand{\up}{\mathrm{e}^{+}}
\newcommand{\down}{\mathrm{e}^{-}}

\begin{document}

\title{Segments and Hilbert schemes of points}

\title{Segments and Hilbert schemes of points}

\author[F.~Cioffi]{Francesca Cioffi}
\address{Francesca Cioffi \\ Universit\`a degli Studi di Napoli \\ Dipartimento di Matematica e Applicazioni \\
Complesso univ. di Monte S. Angelo \\ Via Cintia, 80126 Napoli \\ Italy}
\email{\href{mailto:francesca.cioffi@unina.it}{francesca.cioffi@unina.it}}
\urladdr{\url{http://wpage.unina.it/cioffifr/}}

\author[P.~Lella]{Paolo Lella}
\address{Paolo Lella \\ Universit\`a degli Studi di Torino \\ Dipartimento di Matematica \\ Via Carlo Alberto 10 \\ 10123 Torino \\ Italy}
\email{\href{mailto:paolo.lella@unito.it}{paolo.lella@unito.it}}
\urladdr{\url{http://www.dm.unito.it/dottorato/dottorandi/lella/}}

\author[M.~G.~Marinari]{Maria Grazia Marinari}
\address{Maria Grazia Marinari \\ Universit\`a degli Studi di Genova \\ Dipartimento di Matematica \\ Via Dodecaneso 35 \\ 16146 Genova \\ Italy}
\email{\href{mailto:marinari@dima.unige.it}{marinari@dima.unige.it}}
\urladdr{\url{http://www.dima.unige.it/~marinari/}}

\author[M.~Roggero]{Margherita Roggero}
\address{Margherita Roggero\\ Universit\`a degli Studi di Torino \\ Dipartimento di Matematica \\ Via Carlo Alberto 10 \\ 10123 Torino \\ Italy}
\email{\href{mailto:margherita.roggero@unito.it}{margherita.roggero@unito.it}}
\urladdr{\url{http://www2.dm.unito.it/paginepersonali/roggero/}}

\subjclass[2000]{13P10, 13C99, 14C05}
\keywords{Borel ideal, segment ideal, Hilbert scheme of points, Gotzmann number, Groebner stratum, revlex term order}

\begin{abstract}
Using results obtained from the study of homogeneous ideals sharing the same initial ideal with respect to some term order, we prove the singularity of the point corresponding to a segment ideal with respect to the revlex term order in the Hilbert scheme of points in $\mathbb{P}^n$. In this context, we look inside properties of several types of \lq\lq segment\rq\rq\ ideals that we define and compare. This study led us to focus our attention also to connections between the shape of generators of Borel ideals and the related Hilbert polynomial, providing an algorithm for computing all saturated Borel ideals with the given Hilbert polynomial.  
\end{abstract}

%

\maketitle


\section*{Introduction}

The Hilbert scheme can be covered by some particular affine varieties \cite{CF,FR,NS,RT,LR} that 
have been called \emph{Gr\"obner strata} in \cite{LR} because they are computed from a monomial 
ideal by Gr\"obner basis techniques. The behaviour of Gr\"obner strata can give interesting 
information on the Hilbert scheme itself. 
Very recently, in \cite{R} M. Roggero showed that an open covering of the Hilbert scheme can be constructed 
from Borel ideals by avoiding to involve any term order, that is instead needed for Gr\"obner strata. This 
fact gives us further reasons to investigate Borel ideals and their very particular features.

Among Borel ideals there are several types of \lq\lq segment\rq\rq\  ideals, the definitions of which are 
already well known or arise from some interesting properties of Gr\"obner strata studied in \cite{LR} 
(Definitions \ref{segment-ideal} and \ref{def:segments}). In section 3 we characterize the existence of this 
kind of ideals in terms of the corresponding Hilbert polynomial, in some cases. In this context we need to 
focus our attention also to the shape of admissible polynomials. 

In \cite{H66} the coefficients of Hilbert polynomials are completely characterized by the numbers of 
components of certain subschemes defined by very particular ideals called \emph{fans}. In \cite{RA} these 
numbers of components are described by the shape of minimal generators of Borel ideals. Although the geometric 
meaning is contained in the fans, in section 4 we observe that this connection between coefficients of Hilbert 
polynomials and minimal generators of Borel ideals can be described without using fans, but directly by the 
combinatorial properties of Borel ideals themselves. This study led us to project an algorithm to compute all 
saturated Borel ideals with a given Hilbert polynomial. In section 5 we describe this procedure. 

In \cite{RS} and in \cite{MS} the smoothness of points of Hilbert schemes is studied by means of the dimension of the vector space of the global sections of the normal sheaf to the corresponding projective subscheme. In section 6, applying results of \cite{LR} about Gr\"obner strata, we make some new 
consideration (Theorem \ref{singular point}) on smoothness of points in the Hilbert scheme $\hilbd$ and, 
in particular, prove the main result of this paper, i.e. the point of $\hilbd$ corresponding to the segment 
ideal with respect to the revlex term order is singular (Theorem \ref{revlexpoint}). In literature 
we have not found any proof of such a result.


\section{General setting}

Let $K$ be an algebraically closed field of characteristic $0$, $S:=K[x_0,\ldots,x_n]$ the ring 
of polynomials over $K$ in $n+1$ variables endowed so that $x_0<x_1<\ldots <x_n$ and 
$\mathbb P^n_k=\textnormal{Proj}\, S$ the $n$-dimensional projective space over $K$.

A {\em term} of $S$ is a power product $x^{\alpha}:=x_0^{\alpha_0} x_1^{\alpha_1}\ldots x_n^{\alpha_n}$,
 where ${\alpha_0}, {\alpha_1},\ldots,\alpha_n$ are non negative integers. We set 
$\min(x^{\alpha}):=\min\{i : \alpha_i\not=0\}$ and $\max(x^{\alpha}):=\max\{i : \alpha_i\not=0\}$. We also 
let $\mathbb T:=\{x_0^{\alpha_0} x_1^{\alpha_1}\ldots x_n^{\alpha_n} \ \vert \ (\alpha_0, \alpha_1\ldots,\alpha_n) \in \mathbb N^{n+1}\}$ 
be the monoid of all terms of $S$.

A {\em graded structure} on $S$ is defined by assigning a {\em weight-vector} $w=(w_0,\ldots,w_n)\in\mathbb R_+^{n+1}$ 
and imposing $v_w(x^\alpha)=\sum\limits_{i=0}^n w_i\alpha_i.$ For each non negative integer $t$, $S_t$ is the 
$K$-vector space spanned by $\{x^\alpha\in\mathbb T: \, v_w(x^\alpha)=t\}$. The {\em standard grading} 
corresponds to $w=(1,\ldots,1)$ and, unless otherwise specified, we will always consider it.

For any $N\subseteq \mathbb T$, $N_t$ is the set of the $t$-degree elements of $N$ and 
$\lambda_{i,t}(N):=\vert \{x^{\alpha} \in N_t : i\leq \min(x^{\alpha})\} \vert$ denotes the cardinality of the 
subset of terms of $N_t$ which are not divided by $x_0,\ldots,x_{i-1}$. 
For any homogeneous ideal $I\subseteq S$, $I_t$ is the vector space of the homogeneous polynomials in $I$ of 
degree $t$; $I_{\leq t}$ and $I_{\geq t}$ are the ideals generated by the homogeneous polynomials of $I$ of 
degree $\leq t$ and $\geq t$, respectively.

Fixed any {\em term-order} $\preceq$ on $\mathbb T,$ each $f\in S$ has a unique ordered representation 
$f=\sum\limits_{i=1}^s c(f,\tau_i)\tau_i$, where $c(f,\tau_i)\in K^*$, $\tau_i\in\mathbb T, \tau_1\succ \cdots\succ\tau_s$, 
$T(f):=\tau_1$ is the {\em maximal term} of $f$. For any $F\subset S, \, T\{F\}:=\{T(f) : \, f\in F\}$,  
$T(F):=\{\tau T(f) : \, f\in F, \tau\in\mathbb T\}$ and $\cn(F):=\mathbb T\setminus T(F)$. For any ideal 
$I\subset S, \, T\{I\}=T(I)$ and $\mathcal N(I)$ is an
{\em order ideal}, often called \emph{sous-\'escalier}  or \emph{Gr\"obner-\'escalier} of $I$. 
A subset $G\subset I$ is a \emph{Gr\"obner-basis} of $I$ if $T(G)=T(I)$ (see for instance \cite{Mo}). 

For a monomial ideal $I$, $G(I)$ denotes the unique set of minimal generators of $I$ consisting of terms.

\begin{remark}\label{term orders} \rm 
(1) In our setting, we consider on $\mathbb T$ mainly (standard) graded term orders. In particular, given two 
terms $x^{\alpha}$ and $x^{\beta}$ of $\mathbb T$ of the same degree $t$, we say that $x^{\alpha}$ is less 
than $x^{\beta}$ with respect to: 
\begin{itemize}
\item[(i)] {\em lex} order if $\alpha_k<\beta_k$, where $k=\max\{i\in\{0,\ldots,n\} : \alpha_i\not= \beta_i \}$; 
\item[(ii)] {\em revlex} order if $\alpha_h>\beta_h$, where $h=\min\{i\in\{0,\ldots,n\} : \alpha_i\not= \beta_i \}$;
\item[(iii)] a {\em reverse} order if $\alpha_0>\beta_0$ or $\alpha_0=\beta_0$ and $\dfrac{x^{\alpha}}{x_0^{\alpha_0}}\preceq \dfrac{x^{\beta}}{x_0^{\beta_0}}$, where $\preceq$ is any graded term order on $\mathbb T\cap K[x_1,\ldots,x_n]$.
Recall that a reverse order is well suited for the homogeneization of a Gr\"obner basis \cite{CLO} and that 
revlex is a particular reverse order.
\end{itemize}

(2) With respect to the term orders introduced in (1), for every positive integers $j$ and $n\geq 2$, 
$\mathbb T_j$ can be decomposed in increasing order as follows, where we let
 $\mathbb T(n):=\mathbb T \cap K[x_0,\ldots,x_{n-1}]$:
\begin{itemize} 
 \item[] \emph{lex:}
\begin{equation*}\tag{$\star$}
\mathbb T_j = \mathbb T(n)_j\sqcup \mathbb T_{j-1}x_n = \mathbb T(n)_j\sqcup x_n[\mathbb T(n)_{j-1} \sqcup\mathbb T(n)_{j-2}x_n]= \ldots = \bigsqcup\limits_{i=0}^{j}\mathbb T(n)_{j-i}x_n^i
\end{equation*}
\item[] \emph{revlex:}
\begin{equation*}\tag{$\star\star$}
\begin{split}
\mathbb T_j &= x_0\mathbb T_{j-1}\sqcup \{\tau\in\mathbb T_j: \min(\tau)\geq1\} = 
 x_0\mathbb T_{j-1}\sqcup x_1\{\tau\in\mathbb T_{j-1}: \min(\tau)\geq1\} \sqcup \\
 & \sqcup \{\tau\in\mathbb T_j: \min(\tau)\geq2\}= \ldots = \bigsqcup\limits_{i=0}^{n}x_i \{\tau\in\mathbb T_{j-1}: \min(\tau)\geq i\}
\end{split}
\end{equation*}
\item[] \emph{reverse:}
\begin{equation*}\tag{$\star\star\star$}
\mathbb T_j = x_0\mathbb T_{j-1}\sqcup \{\tau\in\mathbb T_j: \min(\tau)\geq1\}
\end{equation*}
\end{itemize}

(3)\label{segrl}
For each positive integer $j$ and $\omega \in \{1, \ldots , \binom{n+j}{j} -1 \}$, let 
$\Lambda_{\omega}:=\Lambda_{\omega , j}$ be the set of the $\omega$ smallest terms of $\mathbb T_j$ w.r.t. 
revlex order. Thus, from $(\star \star)$ of (2) it follows straightforward that
\[
\Lambda_{\omega} = \bigsqcup\limits_{i=0}^{\gamma(\omega)} x_i\{\tau\in\mathbb T_{j-1}: \, \min(\tau)\geq i\}\sqcup x_{\gamma(\omega)+1}\{\tau_1,\ldots,\tau_{\beta(\omega)}\}
\]
where, $\gamma(\omega)+1=\min\{ t\in \mathbb N : \omega \leq \sum\limits_{\ell=0}^{t} \binom{j-1+n-\ell}{j-1}\}$, 
$\beta(\omega):= \omega - \sum\limits_{\ell =0}^{\gamma (\omega)} \binom{j-1+n-\ell}{j-1}$ and $\tau_1,\ldots \tau_{\beta(\omega)}$ 
are the smallest $\beta(\omega)$ terms of $\mathbb T_{j-1}$ satisfying $\min(\tau_i)\geq \gamma(\omega)+1$, 
for every $1\leq i\leq \beta(\omega)$.

(4) Fixed any term order $\preceq$ on $\mathbb T$ and any weight vector $w$, the {\em weighted term order} 
$\preceq_w$ is defined as follows:
\[
x^\alpha\prec_w x^\beta \, \, \textrm{if} \, \, v_w(x^\alpha)<v_w(x^\beta) \, \, \textrm{or} \, \, v_w(x^\alpha)=v_w(x^\beta) \, \, \textrm{and} \, \, x^\alpha\prec x^\beta.
\]
Speaking of $w$-term order we understand $\preceq$ to be the graded lex order.
\end{remark}

Let $I\subset S$ be any homogeneous ideal. Then, $H_{S/I}(t)$ denotes the Hilbert function of the graded algebra 
$S/I$. It is well known that there are a polynomial $p_{S/I}(z)\in \mathbb{Q}[z]$, called 
\emph{Hilbert polynomial}, and positive integers $\rho_H:=\min\{t\in \mathbb N \ \vert \ H_{S/I}(t')=p_{S/I}(t'), \forall \ t'\geq t\}$,
$\alpha_H:=\min\{t\in \mathbb N \ \vert \ H_{S/I}(t) < \binom{n+t}{t} \}$ called respectively 
\emph{regularity of the Hilbert function} $H$ and \emph{initial degree of $H$} (or also of $I$). For convenience,
 we will also say that $p_{S/I}(z)$ is the Hilbert polynomial {\em for} $I$ or that $I$ is an ideal {\em with} 
Hilbert polynomial $p_{S/I}(z)$. If $I$ is not Artinian, set $\Delta H_{S/I}(t):= H_{S/I}(t)-H_{S/I}(t-1)$, for $t>0$, 
and $\Delta H_{S/I}(0):=1$; we use an analogous notation for Hilbert polynomials. If $h$ is a linear form 
general on $S/I$, then it is easy to prove that $p_{{S}/{(I,h)}}=\Delta p_{{S}/{I}}$. 

Polynomials $p(z)\in \mathbb Q[z]$ that are Hilbert polynomials of projective subschemes are called 
{\em admissible} and are completely characterized in \cite{H66} by the fact that they can be always written in 
a unique form of the following type (see \cite{H66,M}), where $\ell$ is the degree of $p(z)$ and $m_0\geq m_1\geq \cdots\geq m_{\ell}\geq 0$ are integers: 
\[
p(z)=\sum\limits_{i=0}^{\ell} \binom{z+i}{i+1} - \binom{z+i-m_i}{i+1}.
\]

The saturation of a homogeneous ideal $I\subset S$ is $I^{sat}:=\{f \in S \ \vert \ \forall \ i\in {0,\ldots,n}, \exists\ k_i : x_i^{k_i}f \in I\}= \cup_{h\geq 0} (I:\gm^h)$, 
where $\gm=(x_0,\ldots,x_n)$, and $I$ is saturated if $I=I^{sat}$.

If $X\subset \mathbb{P}^n_K$ is a projective subscheme, $reg(X)$ is its {\em Castelnuovo-Mumford regularity}, 
i.e $reg(X)=\min\{t\in \mathbb N\ \vert \ H^i(\mathcal I_X(t'-i))=0, \forall \ t'\geq t \}$. 

An ideal $I\subset S$ is $m$-regular if the $i$-th syzygy module of $I$ is generated in degree $\leq m+i$ and 
the regularity $reg(I)$ of $I$ is the smallest integer $m$ for which $I$ is $m$-regular. If $I$ is saturated 
and defines a scheme $X$, then $reg(I)=reg(X)$ and we set $H_X(t):=H_{S/I}(t)$ and $p_X(z):=p_{{S}/{I}}(z)$. 

For an admissible polynomial $p(z)$, the {\em Gotzmann number} $r$ is the best upper bound for the 
Castelnuovo-Mumford regularity of a scheme having $p(z)$ as Hilbert polynomial and is computable by using the 
following unique form of an admissible polynomial:
\[
p(z)=\binom{z+a_1}{a_1}+\binom{z+a_2-1}{a_2}+\ldots+\binom{z+a_r-(r-1)}{a_r},
\]
where $a_1\geq a_2\geq \ldots \geq a_r\geq 0$. We refer to \cite{Gr} for an overview of these arguments. 

\begin{example}\label{numeroGgrado1}\rm  
If $p(z)=dz+1-g$ is an admissible polynomial, then its Gotzmann number is $r=\binom{d}{2}+1-g$. Indeed, we get 
\[
p(z)=\binom{z+1}{1}+\ldots+\binom{z+1-(d-2)}{1}+\binom{z+0-(d-1)}{0}+\ldots+\binom{z+0-\binom{d-2}{2}+g}{0}.
\] 
\end{example}


\section{Results on Borel ideals and Gr\"obner strata}

\begin{definition}\rm 
(1) For any $x^{\alpha}\in \mathbb{T}$ such that $\alpha_j>0$, the terms obtained from $x^{\alpha}$ via a 
$j$-th {\em elementary move} are:
\begin{itemize}
\item[(i)] $e_j^+(x^{\alpha}):=x_0^{\alpha_0}\ldots x_j^{\alpha_j-1}x_{j+1}^{\alpha_{j+1}+1}\ldots x_n^{\alpha_n}$,
 for any $j\in \{0,\ldots,n-1\}$; 
\item[(ii)] $e_j^-(x^{\alpha}):=x_0^{\alpha_0}\ldots x_{j-1}^{\alpha_{j-1}+1} x_j^{\alpha_j-1}\ldots x_n^{\alpha_n}$, 
 for any $j\in \{1,\ldots,n\}$,
\end{itemize}
and for each positive integer $a$ we will denote by $(\down_j)^a,(\up_{j})^a$ the corresponding elementary move 
applied $a$ times. 

(2) For any positive integer $t$, $<_B$ denotes the partial order on $\mathbb{T}_t$ given by the transitive 
closure of the relation: \ 
$\down_{j}(x^{\beta}) < x^{\beta}$, i.e $x^{\alpha} <_B x^{\beta}$ if $x^{\alpha}$ is gotten from $x^{\beta}$ via a 
finite sequence of elementary moves $\down_{j}$. 

(3) A set $B\subset \mathbb T_t$ is a {\em Borel set} \ if, for every $x^{\alpha}$ of $B$ and $x^{\beta}$ of 
$\mathbb{T}_t$, $x^{\alpha} <_B x^{\beta}$ implies that $x^{\beta}$ belongs to $B$.

(4) A monomial ideal $J\subset S$ is a {\em Borel ideal} \ if, for every degree $t$, $J\cap \mathbb T_t$ is a 
Borel set.  
\end{definition}

The combinatorial property by which Borel ideals are defined is also called {\em strong stability}. It has been 
first introduced in \cite{Gu} and later in \cite{H66}, where the ideals satisfying it are called {\em balanced}. 
In characteristic $0$ it is equivalent to the property for an ideal $J$ of being fixed by lower triangular 
matrices, from which the name Borel ideals derives.

From the definition it follows immediately that, if $B\subset \mathbb T_t$ is a Borel set, then the set 
$N:=\mathbb T_t\setminus B$ has the  property that for every $x^{\gamma}\in N$ and $x^{\delta}\in\mathbb T_t$, 
with $x^{\delta} <_B x^{\gamma}$ it holds $x^{\delta}\in N$, that is $N$ is closed w.r.t. elementary moves 
$\down_j$. In particular, if $J$ is a Borel ideal, then for every integer $t\geq 0$, $\mathcal N(J)_t$ is 
closed w.r.t. elementary moves $\down_j$ and $J_t$ is closed w.r.t. elementary moves $\up_j$.

\begin{remark}\label{cipa}\rm
Note that, for every term order $\preceq$, if $x^{\alpha},x^{\beta}\in\mathbb T_t$ satisfy 
$x^{\alpha}<_B x^{\beta}$ then $x^{\alpha}\prec x^{\beta}$. Namely, as $x^{\alpha}<_B x^{\beta}$ means that 
there is a finite number of elementary moves $\down_j$ connecting $x^{\beta}$ to $x^{\alpha}$, assuming that 
$x_j\mid x^{\beta}$ for a suitable $0\leq j\leq n$, we can verify our contention for $x^{\alpha}=\down_j(x^{\beta})$. 
Setting $\tau:=\frac{x^{\beta}}{x_j}$ and writing $x^{\alpha}=\down_j(x^{\beta})=x_{j-1}\tau, \, x^{\beta}=x_j\tau,$ 
we get $x^{\alpha}\prec x^{\beta}$ as $x_{j-1}\prec x_j$. 
\end{remark}

\begin{proposition}\label{reeves}
For a Borel ideal $J\subset S$, 
\begin{itemize}
\item[(i)] in our notation $J^{sat}$ is obtained by setting $x_0=1$ in the minimal generators of $J$;
\item[(ii)] \label{HS} the Krull dimension of ${S}/{J}$ is equal to $\min\{\max(x^{\alpha}) : x^{\alpha} \in J\}=\min\{ i\in \{0,\ldots,n\} : x_i^t \in J, \text{ for some } t\}$;
\item[(iii)] the regularity of $J$ is equal to the maximum degree of its minimal generators.
\end{itemize}  
\end{proposition}

\begin{proof}
(i) For example, see \cite[Property 2]{RA}. 

(ii) This result follows straightforward from Lemma 3.1(a) of \cite{HS} or from Corollary 4, section 5, 
chapter 9 of \cite{CLO}. Thus, if $J$ is saturated and $\ell$ is the degree of the Hilbert polynomial of $J$, 
we get $\ell=\min\{ i\in \{0,\ldots,n\} : x_i^t \in J, \text{ for some } t\}-1$.

(iii) See \cite[Proposition 2.9]{BS}.
\end{proof}

\begin{remark}\label{espansione} \rm
Let $B\subset \mathbb T_t$ be a non-empty Borel set, $N:=\mathbb T_t\setminus B$ and $J=(B)$ the Borel ideal 
generated by the terms of $B$, so that $\mathcal N(J)_t=N$. Thus 
\begin{eqnarray*}
& \mathcal N(J)_{t+1}= x_0 N \sqcup x_1 \{x^{\alpha}\in N : 1\leq \min(x^{\alpha})\} \sqcup & \\
& \sqcup x_2\{x^{\alpha} \in N : 2\leq \min(x^{\alpha})\} \sqcup \ldots \sqcup x_{n-1} \{x^{\alpha} \in N : n-1\leq \min(x^{\alpha})\} &
\end{eqnarray*}
and $\mathbb T_{t+1}\setminus\mathcal N(J)_{t+1}$ is a Borel set.
In particular, if $J$ is a Borel ideal and $N:=\mathcal N(J)_t$, for every degree $t$ we have 
(see \cite[Theorem 3.7]{MR1}): 
\begin{eqnarray*} 
& \mathcal N(J)_{t+1} \subseteq \mathcal N(J_{\leq t})_{t+1}= x_0 N \sqcup x_1 \{x^{\alpha}\in N : 1\leq \min(x^{\alpha})\} \sqcup & \\
& \sqcup x_2\{x^{\alpha} \in N : 2\leq \min(x^{\alpha})\} \sqcup \ldots \sqcup x_{n-1} \{x^{\alpha} \in N : n-1\leq \min(x^{\alpha})\} &
\end{eqnarray*}
from which $\mid\mathcal N(J)_{t+1}\mid \leq \mid \mathcal N(J_{\leq t})_{t+1}\mid = \sum\limits_{i=0}^{n-1}\lambda_{i,t}(N(J_{\leq t}))$ 
and $G(J)_{t+1}=\mathcal N(J_{\leq t})_{t+1} \setminus \mathcal N(J)_{t+1}$.
\end{remark}

\begin{definition}\rm 
For each Borel subset $A\subset \mathbb T_t$ the {\em minimal} elements of $A$, w.r.t. $<_B$,  are the terms 
$x^{\alpha}\in A$ such that  $\down_j(x^\alpha)\notin A$ for every
$j>0$ with $\alpha_j>0$ and the {\em maximal} elements outside $A$, w.r.t. $<_B$, are the terms 
$x^{\beta}\notin A$ such that $\up_j(x^\beta)\in A$ for every $j>0$ with $\beta_j>0$.
\end{definition}

\begin{remark}\label{cicca}\rm
Let $B\subset \mathbb T_t$ be a Borel subset, if $x^{\alpha}\in B, x^{\beta}\notin B$ are respectively a minimal 
element of $B$ and a maximal element outside $B$ w.r.t. $<_B,$ then both
$B\setminus \{x^\alpha\}$ and $B\cup\{x^\beta\}$ are Borel subsets of $\mathbb T_t$ as by definition both are 
closed w.r.t. elementary moves $\up_{j}$.  
\end{remark} 

\begin{proposition}\label{cicca1}
Let $p(z)$ be an admissibile polynomial with Gotzmann number $r$ and let $J\subset S$ be a Borel ideal with 
$p(z)$ as Hilbert polynomial. Then, for each $s>r$, a minimal term of $J_s$ w.r.t. $<_B$ is divided by  $x_0$.  
\end{proposition}

\begin{proof}
As $s>r\geq reg(J)$, for each $x^\alpha\in J_s$ it holds $x^\alpha=\tau\cdot x^\gamma$
for some $x^\gamma \in G(J)$ and $\tau\in \mathbb T$ with $\deg(\tau)>0$. If $x_0\nmid x^\alpha$, let $j>0$ be 
such that $x_j\mid\tau$, then $x^{\alpha'}:=\tau'\cdot x^\gamma$, with $\tau'=\frac{\tau\cdot x_0}{x_j}$, 
satisfies $x^{\alpha'}\in J_s$ and $x^{\alpha'}<_B x^\alpha,$ contradicting the minimality of $x^\alpha$.
\end{proof}

Given an admissible polynomial $p(z)$, a term order $\preceq$ and a monomial ideal $J$ with $p(z)$ as Hilbert 
polynomial, the {\em Gr\"obner stratum} $\mathcal St(J,\preceq)$ in the Hilbert scheme $\hilbp$ of subschemes 
of $\mathbb P^n$ with Hilbert polynomial $p(z)$ is an affine variety that parameterizes the family of ideals 
having the same initial ideal $J$ with respect to $\preceq$ \cite{RT,FR,LR,CF,NS}. When only homogeneous ideals 
are concerned, we write $\mathcal St_h(J,\preceq)$. Now, we recall briefly the construction of 
$\mathcal St(J,\preceq)$, and hence of $\mathcal St_h(J,\preceq)$, referring to Definition 3.4 of \cite{LR}, 
although here we omit many details that make the procedure more efficient. 

For any term $x^{\alpha}$ of $G(J)$ we set $F_{\alpha}:= x^{\alpha}+\sum_{\{x^{\beta}\in \mathcal N(J)\ :\ x^{\beta} < x^{\alpha}\}} c_{\alpha \beta} x^{\beta}$, 
considering $c_{\alpha\beta}$ as new variables. Then, we reduce all the $S$-polynomials $S(F_{\alpha},F_{\alpha'})$ 
with respect to $\{F_{\alpha}\}_{x^{\alpha}\in J}$. The ideal $\mathcal A(J)$ generated in 
$K[c_{\alpha \beta}]$ by the $x$-coefficients of the reduced polynomials is the defining ideal of 
$\mathcal St(J,\preceq)$ and does not depend on the reduction choises. If in particular we set 
$F_{\alpha}:= x^{\alpha}+\sum_{\{x^{\beta}\in \mathcal N(J)_t\ :\ x^{\beta} < x^{\alpha}\}} c_{\alpha \beta} x^{\beta}$, 
where $t$ is the degree of $x^{\alpha}$, then we obtain the ideal of $\mathcal St_h(J,\preceq)$.

For properties of Gr\"obner strata we refer to \cite{RT,FR,LR}, but it is noteworthy to point out an unexpected 
feature of Gr\"obner strata, i.e. they are homogeneous varieties with respect to some non-standard graduation 
\cite{FR,LR}. Thus, the embedding dimension of $\mathcal St_h(J,\preceq)$, denoted by $ed(\mathcal St_h(J,\preceq))$, 
is the dimension of the Zariski tangent space of the stratum at the origin and can be computed by the same 
procedure which produces Gr\"obner strata. In fact, the ideal $\mathcal L(J)$ generated in $K[c_{\alpha \beta}]$ 
by the linear components of the generators of $\mathcal A(J)$, as computed above, defines the Zariski tangent 
space of the stratum at the origin (Theorems 3.6(ii) and 4.3 of \cite{LR}). This fact gives a new tool for 
studying the singularities of Hilbert schemes. 


\section{Segments}

\begin{definition}\label{segment-ideal}\rm 
A set $B\subset \mathbb T_t$ is a {\em segment} with respect to (w.r.t., for short) a term order $\preceq$ on 
$\mathbb T$ if, whenever a term $\tau$ belongs to $B$, all the $t$-degree terms which are greater than $\tau$ 
belong to $B$. A monomial ideal $I$ is a {\em segment ideal} w.r.t. $\preceq$, if $I\cap \mathbb T_t$ is a 
segment w.r.t. $\preceq$, for every $t\geq 0$.
\end{definition}

\begin{lemma}\label{gradi}
Let $I\subset S$ be a saturated Borel ideal, $\preceq$ any term order on $\mathbb T$ and $p>q$ integers. 
If $I_p$ is a segment then $I_q$ is a segment too. 
\end{lemma}

\begin{proof}
Let $x^{\alpha}$ be a term of $I_q$ and $x^{\beta}$ a term of $\mathbb T_q$ such that 
$x^{\alpha} \preceq x^{\beta}$, hence $x_0^{p-q} x^{\alpha} \preceq x_0^{p-q}x^{\beta}$ and, since $I_p$ is a 
segment, $x_0^{p-q}x^{\beta}$ belongs to $I_p$. Recalling that $I$ is saturated, $x^{\beta}$ belongs to $I_q$ 
and we are done.
\end{proof}

\begin{remark} \label{rem:segment} \rm 
A segment is a Borel set and a segment ideal is a Borel ideal. Indeed, by recovering the arguments of Remark 
\ref{term orders}(2), $x_i x^{\alpha} \prec x_h x^{\alpha}$ if $i<h$, thus $x^{\alpha} <_B x^{\beta}$ implies 
that $x^{\alpha} \prec x^{\beta}$, for any term order $\preceq$. In particular, if $\preceq$ is the lex order 
and $I$ is a monomial ideal generated in degree $\leq q$ such that $I_q$ is a segment w.r.t. $\preceq$, then 
$I_p$ is a segment too, for every $p>q$.
\end{remark}

\begin{remark} \label{rem:segmentcic} \rm
(1) To each admissible polynomial $p(z)$ of degree $0\leq\ell\leq n$ corresponds a unique saturated segment ideal 
$L(p(z))$ w.r.t. lex order (see \cite{B,M}). In particular for a constant polynomial $p(z)=d$ we have the 
following, where $\mathbb T(2):=\mathbb T \cap K[x_0,x_1]$, 
\[
\begin{split}
 L(d)&=(x_n,x_{n-1},\ldots,x_2,x_1^d),\\
 \cn(L(d))_j &=\begin{cases} 
\mathbb T (2)_j  & \text{\ if\ }    0\leq j<d \cr \{x_0^{d+i},x_0^{d+i-1}x_1,\ldots,x_0^{i+1}x_1^{d-1}\} & \text{\ if \ }  j=d+i,\ \forall\ i \geq 0, \end{cases}.
\end{split}
\]

(2) A segment ideal w.r.t. the revlex order exists if and only if the Hilbert polynomial is constant and the 
Hilbert function $H$ is non-increasing, i.e. $\Delta H(t)\leq 0$ for every $t>\alpha_H=\min\{t\in \mathbb N \vert H(t)<\binom{t+n}{n} \}$ 
\cite{D,MR99}.

(3) The same reasoning of \cite{D,MR99} shows that, more in general, a segment ideal $J$ w.r.t. a reverse term 
order exists if and only if the Hilbert polynomial is constant and the Hilbert function $H$ is non-increasing. 
Namely, if $\alpha_H$ is the initial degree and $x_1^{\alpha_H}\in \cn(J)$ it must be  $x_1^{\alpha_H+1}\in J$ 
otherwise, letting $\tau\succ x_1^{\alpha_H}$ be the smallest degree $\alpha_H$ term in $J$, it would be 
$x_1^{\alpha_H+1}\in \cn(J)$ with $x_1^{\alpha_H+1}\succ x_0\tau\in J$.
\end{remark}

\begin{proposition}\label{criterio} 
If an ideal $J\subset S$ of initial degree $\alpha_H$ has the property that there exist an integer 
$t\geq \alpha_H$ and four terms $x^{\alpha}, x^{\beta}\in \mathcal N(J)_t$, \, $x^{\gamma},x^\delta \in J_t$ with 
$x^{\alpha+\beta}=x^{\gamma+\delta}$, then $J$ is not a segment ideal w.r.t. any term order $\preceq$.
\end{proposition}

\begin{proof}
If $J$ were a segment ideal w.r.t some $\preceq$, by the given assumptions we would have in particular both 
$\cn(J)_t\ni x^{\beta}\prec x^{\delta}\in J_t$ and $\cn(J)_t\ni x^{\alpha}\prec x^{\gamma}\in J_t$. From these
 it would follow $x^{\alpha+\beta}\prec x^{\alpha+\delta}\prec x^{\gamma+\delta}$ contradicting 
$x^{\alpha+\beta}=x^{\gamma+\delta}$.
\end{proof}

\begin{example}\label{ciccimargherita}\rm (1) The (saturated) Borel ideal 
$J=(x_2^3,x_1^3 x_2^2,x_1^5 x_2,x_1^6)\subset K[x_0,x_1,x_2]$ is not a segment ideal w.r.t. any term order as 
it satisfies the conditions of Proposition \ref{criterio}. Namely, its initial degree is $3$ and, for 
$t=6\geq3$, we have: $J_6 \ni x_0^3x_2^3, x_1^6$ and
 $x_0^2x_1^2x_2^2, x_0x_1^4x_2\in \mathcal N(J)_6$ with 
$x_0^3x_2^3\cdot x_1^6=x_0^2x_1^2x_2^2 \cdot x_0x_1^4x_2$.
 
(2) Here we show that Proposition \ref{criterio} cannot be inverted. The Borel ideal 
$J=(x_2^3,x_1 x_2^2,x_1^2 x_2,$ $x_0^2x_2^2,x_0^3x_1x_2,x_0^5x_2,x_1^7)\subset K[x_0,x_1,x_2]$ 
of \cite[Example 5.8]{CS} has the property that $J_3$ is a segment w.r.t. revlex order 
while $J_t$ is a segment w.r.t. lex order, for every $t\geq 4$ so that at each degree it 
does not satisfy the conditions of Proposition \ref{criterio}. Nevertheless $J$ is not a 
segment w.r.t. any term order $\preceq$, namely if it were, from $x_1^2 x_2\in J_3$ and 
$x_0 x_2^2\in\mathcal N(J)_3$, it would follow $x_0x_2\prec x_1^2$ contradicting 
$(x_0x_2)^2\in J_4, x_1^4\in\mathcal N(J)_4$. Note also that $J^{sat}=(x_2,x_1^7)$ is the 
saturated lex segment.
\end{example}

\begin{definition}\label{def:segments}\rm 
Let $I\subset S$ be a non null saturated Borel ideal and $\preceq$ a term order on $\mathbb T$. 
\begin{itemize}
\item[(a)]\cite{LR} $I$ is a {\em hilb-segment ideal} if $I_r$ is a segment, where $r$ is the Gotzmann number 
 of the Hilbert polynomial of $I$; 
\item[(b)] $I$ is a {\em reg-segment ideal} if $I_{\delta}$ is a segment, where $\delta$ is the regularity of $I$;
\item[(c)] $I$ is a {\em gen-segment ideal} if, for every integer $s$, $G(I)_s$ consists of the greatest terms among the 
$s$-degree terms not in $\langle I_{s-1}\rangle$. 
\end{itemize}
\end{definition}

\begin{remark}\label{rem:criterion}\rm
The criterion given by Proposition \ref{criterio} can be adapted also to hilb-segment ideals and to reg-segment 
ideals $I$, by simply verifying it at degree $r=$ Gotzmann number and at degree $\delta=reg(I)$, respectively. 
Computational evidence suggests that the condition of this criterion is also necessary for reg-segment and hilb-segment ideals.
\end{remark}

The following results about Gr\"obner strata motivate the definitions of reg-segment ideal and of hilb-segment 
ideal, respectively. 

\begin{proposition}\label{Gstrata} 
(i) Let $I\subset S$ be a Borel saturated ideal generated in degree $\leq r$. If $s$ is the maximum degree of 
terms in $G(I)$ in which $x_1$ appears, then $\mathcal St_h(I_{\geq m})\cong\mathcal St_h(I_{\geq s})$, 
for every $m\geq s$. In particular, if $x_1$ does not appear in any term of $G(I)$, then 
$\mathcal St_h(I_{\geq m})\cong\mathcal St_h(I_{\geq s})$ for every $m$ (Theorem B and, also, Corollary 4.8(ii),(iii) of \cite{LR}).

(ii) An isolated, irreducible component of $\hilbp$ that contains a smooth point corresponding to a hilb-segment
 ideal is rational (Theorem D and, also, Corollary 6.8 of \cite{LR}).
\end{proposition}  

\begin{proposition}\label{prop:segments}
Let $I\subset S$ be a saturated Borel ideal and $\preceq$ a term order on $\mathbb T$. Then
\begin{itemize}
\item[(i)] $I$ segment ideal $\Rightarrow$ $I$ hilb-segment ideal $\Rightarrow$ $I$ reg-segment ideal 
           $\Rightarrow$ $I$ gen-segment ideal. 
\item[(ii)] $\preceq$ is the lex order $\Leftrightarrow$ the implications in (i) are all equivalences, 
            for every ideal $I$.
\item[(iii)] If the projective scheme defined by $I$ is $0$-dimensional, then: $I$ segment ideal 
             $\Leftrightarrow$ $I$ hilb-segment ideal $\Leftrightarrow$ $I$ reg-segment ideal.
\end{itemize}
\end{proposition}

\begin{proof}
(i) The first implication is obvious. For the second one, it is enough to apply Lemma \ref{gradi} since 
$r\geq \delta$. For the third implication, recall that $I$ is generated in degrees $\leq \delta$, by definition. 
Moreover, if $I$ is a reg-segment ideal, by Lemma \ref{gradi} $I_t$ contains the greatest terms of degree $t$, 
for every $t\leq \delta$. Thus, in particular, minimal generators of $I$ must to be the greatest possible. 

(ii) First, suppose that $\preceq$ is the lex order. Then, by (i), it is enough to show that a gen-segment ideal 
is also a segment ideal. Indeed, by induction on the degree $s$ of terms and with $s=0$ as base of induction, 
for $s>0$ suppose that $I_{s-1}$ is a segment. Thus, by Remark \ref{rem:segment}, we know that $\langle I_{s-1}\rangle_s$ 
is a segment and, since possible minimal generators are always the greatest possible, we are done.

Vice versa, if $\preceq$ is not the lex order, let $s$ be the minimum degree at which the terms are ordered in 
a different way from the lex one. Thus, there exist two terms $x^{\alpha}$ and $x^{\beta}$ with maximum 
variables $x_l$ and $x_h$, respectively, such that $x^{\beta}\prec x^{\alpha}$ but $x_h \succ x_l$. The ideal 
$I=(x_h,\ldots,x_n)$ is a gen-segment ideal but not a segment ideal, since $x^{\beta}$ belongs to $I$ and 
$x^{\alpha}$ does not.

(iii) It is enough to show that, in the $0$-dimensional case, a reg-segment ideal $I$ is also a segment ideal. 
By induction on the degree $s$, if $s\leq \delta$, then the thesis follows by the hypothesis and by Lemma 
\ref{gradi}. Suppose that $s>\delta$ and that $I_{s-1}$ is a segment. At degree $s$ there are not minimal 
generators for $I$ so that a term of $I_s$ is always of type $x^{\alpha}x_h$ with $x^{\alpha}$ in $I_{s-1}$. 
Let $x^{\beta}$ be a term of degree $s$ such that $x^{\beta}\succ x^{\alpha}x_h$, thus $x^{\beta}\succ x^{\alpha}x_0$. 
By Proposition \ref{HS}, we have that $(x_1,\ldots,x_n)^s \subseteq I$. So, if $x^\beta$ is not divided by $x_0$, 
then $x^\beta$ belongs to $I_s$, otherwise there exists a term $x^\gamma$ such that $x^\beta=x^\gamma x_0$. 
Thus $x^\gamma \succ x^\alpha$ and by induction $x^\gamma$ belongs to $I_{s-1}$ so that $x^\beta=x^\gamma x_0$ 
belongs to $I_s$.
\end{proof}

\begin{example}\rm 
Let $\preceq$ be the revlex order. 

(1) The ideal $I=(x_2^2,x_1x_2)\subset K[x_0,x_1,x_2]$ is a hilb-segment ideal, but it is not a segment ideal. 
In this case, the Hilbert polynomial is $p(z)=z+2$ with Gotzmann number $2$ and $reg(I)=2$. We have 
$x_1^3\in\mathcal N(I)$ and $x_0x_1x_2\in I$ with $x_1^3\succeq x_0x_1x_2$.

(2) $I'=(x_2^3,x_1 x_2^2,x_1^2 x_2)\subset K[x_0,x_1,x_2]$ is a reg-segment ideal, but not a hilb-segment ideal. 
In this case, the Hilbert polynomial is $p(z)=z+4$ with Gotzmann number $4$ and $reg(I')=3$. We get 
$x_0 x_2^3\in I'$ with $x_0x_2^3\preceq x_1^4\notin I'$.

(3) $I''=(x_4^2,x_3 x_4,x_3^3)\subset K[x_0,\ldots,x_4]$ is a gen-segment ideal but not a reg-segment ideal. 
In this case, the Hilbert polynomial is $p(z)=2z^2+2z+1$ with Gotzmann number $12$ and $reg(I'')=3$. We get 
$x_0x_4^2\in I''$ with  $x_0x_4^2\preceq x_2^3\notin I''$. 
\end{example}

\begin{remark}\rm
If $I$ is a saturated Borel ideal and also an almost revlex segment ideal, as defined in \cite{D}, then it is 
a gen-segment ideal w.r.t. revlex order. 
\end{remark}

\begin{theorem}\label{grado0} 
To the ideal $J$ generated by a Borel set $B\subset\mathbb T_d$, consisting of all but $d$ terms of degree $d$, 
corresponds a projective scheme with Hilbert polynomial $p(z)=d$.
\end{theorem}

\begin{proof}
By the Borel condition, we have that $x_1^d$ belongs to $J$, otherwise 
$\vert \mathcal N(J)_{d}\vert\geq d+1>d$ by Remark \ref{espansione}(1), so that by Remark \ref{espansione}(2)
we have $\vert \mathcal N(J)_{t}\vert=d$, for every $t\geq d$. The ideal $I=J^{sat}$ is the saturated ideal of 
a projective subscheme with Hilbert polynomial $p(z)=d$.
\end{proof}

\begin{remark}\rm \label{numero generatori}
(1) For every positive integer $d$ and any term order $\preceq$ on $\mathbb T$, there exists a unique saturated 
segment ideal $I\subset S$, with Hilbert polynomial $p(z)=d$. This is a straightforward consequence of Theorem 
\ref{grado0}: it is enough to take the ideal $J$ generated by all but the least $d$ terms of degree $d$.

(2) In \cite{MR99} for the revlex order and then in \cite{CS} for each term order, it is shown that the generic 
initial ideal of a set $X$ of $d$ general points in $\mathbb P^n$ is a segment ideal with Hilbert polynomial 
$p(z)=d$. As the Hilbert function of $X$ is the maximum possible, that is $H_X(t)=\min\{\binom{t+n}{t},d\}$, 
we deduce that this is the Hilbert function of the saturated segment ideal of (1). 

(3) For the revers term orders it is possible to give a direct and constructive proof of (1). 
If $J\subset S$ is a segment ideal w.r.t. a revers order with Hilbert polynomial $p(z)=d$, its Hilbert function 
must to be non-increasing by Remark \ref{rem:segmentcic}(2) and strictly increasing until it reaches the value 
$d$, after which it is always equal to $d$, because $J$ is a saturated ideal of Krull dimension $1$. 
Thus, $H_{S/J}(t)$ must to be the maximum possible and we have two items:
\begin{itemize}
\item[(i)] $\alpha_H=\rho_H+1$, so that $J=(x_1,\ldots,x_n)^{\alpha_H}$;
\item[(ii)] $\alpha_H=\rho_H$, so that $J$ is generated only in degrees $\alpha_H$ and $\alpha_H+1$; 
more precisely, the minimal generators of degree $\alpha_H$ are the greatest $\binom{\alpha_H+n}{\alpha_H}-d$ 
terms of $\mathbb T_{\alpha_H}$ (so that in $\mathcal N(J)_{\alpha_H}$ there are $d-\binom{\alpha_H+n-1}{n-1}$ 
terms $x^{\beta}$ with $\min(x^{\beta})\geq 1$) and the minimal generators of degree $\alpha_H+1$ are the all 
terms $\tau\succeq x_1^{\alpha_H+1}$ which are not multiples of terms in $J_{\alpha_H}$ (these terms are at 
least $d-\binom{\alpha_H-1+n}{\alpha_H-1}$, by Remark \ref{espansione}).
\end{itemize} 
It follows that in case (i) we have $\vert G(J)\vert=\binom{\rho_H+n}{n-1}$ and in case (ii) 
$\vert G(J)\vert\geq \binom{\rho_H+n}{n}-d+d-\binom{\rho_H+n-1}{n}=\binom{\rho_H+n-1}{n-1}$.
\end{remark}

\subsection{On hilb-segment ideals}

Let $\preceq$ be any term order and $p(z)$ an admissible polynomial with Gotzmann number $r$. We want to see 
under which conditions there exists a hilb-segment ideal for $p(z)$. In this context, it is immediate to see 
that, if $r=1$, then $p(z)=\binom{z+\ell}{\ell}$, where $\ell< n$ is the degree of $p(z)$, so that 
$I=(x_{\ell+1},\ldots,x_n)$ is the hilb-segment ideal for $p(z)$. Moreover, we have already observed that a 
hilb-segment ideal exists always for a constant polynomial $p(z)=d$.

\begin{example}\label{exNotSegment}\rm The following saturated Borel ideals are not hilb-segment ideals for any 
term order: 

1) $J=(x_2^2,x_1^3x_2,x_1^4)\subset K[x_0,x_1,x_2]$, (see \cite{CS}) as $H=(1,3,5,7,\ldots,p(z)=7,\ldots)$ 
we have $r=7$ so, if $J$ were a hilb-segment ideal w.r.t. some $\preceq$, at degree $7$ we should have 
$\cn(J)_7\ni x_0^4x_1^2x_2\prec x_0^5x_2^2\in J_7$ and $\cn(J)_7\ni x_0^4x_1^2x_2\prec x_0^3x_1^4\in J_7$ 
contradicting $(x_0^4x_1^2x_2)^2=x_0^5 x_2^2 \cdot x_0^3 x_1^4$.

2) $J=(x_2^3,x_1x_2^2,x_1^2x_2)\subset K[x_0,x_1,x_2],$ as $H=(1,3,6,7,\ldots,p(z)=z+4,\ldots)$ we have $reg(J)=3$ 
and $r=4,$ so we can repeat the same reasoning of 1) with $x_0x_1^2x_2\in J_4$ and $x_1^4$, $x_0^2x_2^2\in 
\mathcal N(J)_4$.

\end{example}

\begin{proposition} 
In $K[x_0,x_1,x_2]$ every saturated Borel ideal with Hilbert polynomial $p(z)=d \leqslant 6$ is a hilb-segment 
ideal. Whereas for every $p(z) = d \geqslant 7$, a saturated Borel ideal, which is not a hilb-segment
for any term order, always exists. 
\end{proposition}  

\begin{proof}
We give a direct constructive proof of the result, based in part on the characterization of the Borel subsets in three variables of \cite{M01}.  
\begin{itemize}
\item[$d \leqslant 2$] there exists a unique saturated Borel ideal $(x_2,x_1^d)$, which is the hilb-segment ideal w.r.t. lex order;
\item[$d=3$] there are only two saturated Borel ideals: the hilb-segment ideals $(x_2,x_1^3)$ (w.r.t. lex) and $(x_2^2,x_1x_2,x_1^2)$ (w.r.t. revlex);
\item[$d=4$] there are only two saturated Borel ideals: the hilb-segment ideals $(x_2,x_1^4)$ (w.r.t. lex) and $(x_2^2,x_1x_2,x_1^3)$ (w.r.t. revlex);
\item[$d=5$] there are three saturated Borel ideals: the hilb-segment ideals $(x_2,x_1^5)$ (w.r.t lex), $(x_2^2,$ $x_1x_2,x_1^4)$ (w.r.t. $(4,2,1)$-term order) and $(x_2^2,x_1^2 x_2,x_1^3)$ (w.r.t revlex);
\item[$d=6$] there are four saturated Borel ideals: the hilb-segment ideals $(x_2,x_1^6)$ (w.r.t. lex),  $(x_2^2,x_1 x_2,x_1^5)$ (w.r.t. $(5,2,1)$-term order), $(x_2^2,x_1^2x_2,x_1^4)$ (w.r.t. $(3,2,1)$-term order) and $(x_2^3,x_1 x_2^2,$ $x_1^2 x_2,x_1^3)$ (w.r.t revlex);
\item[$d \geqslant 7$] 
(i) Let us firstly consider the case $d = 2a+1, a \geq 3$ and the ideal $J = (x_2^2,x_1^a x_2,$ $x_1^{a+1})$. 
It has Hilbert polynomial $p(z)=2a+1$, in fact in degree $2a+1$, the $2a+1$ monomials 
$\{x_0^{a+i} x_1^{a-i} x_2,x_0^{a+j+1} x_1^{a-j},\ i=1,\ldots,a,\ j=0,\ldots,a\}$, belong to the quotient. 
Moreover $x_0^{2a-1} x_2^2,x_0^a x_1^{a+1} \in J$ and $x_0^{2a-2} x_1^2 x_2, x_0^{a+1} x_1^{a-1} x_2 \notin J$, 
but $x_0^{2a-1} x_2^2 \cdot x_0^a x_1^{a+1} = x_0^{2a-2} x_1^2 x_2 \cdot x_0^{a+1} x_1^{a-1} x_2$ 
(if $a=3$ this is exactly the ideal of Example \ref{exNotSegment} 1)).

\noindent (ii) In the case $d=2a$, $a\geq4$, let us consider the ideal $J = (x_2^3,x_1 x_2^2, x_1^2 x_2, x_1^{2a-3})$. 
It has Hilbert polynomial $p(z)=2a$, 
namely $\cn(J)_{2a}=\{x_0^{2a-2} x_2^2,x_0^{2a-2} x_1 x_2,x_0^{2a-1} x_2,$ $x_0^{2a-i} x_1^i,\ i=0,\ldots,2a-4\}$.
Moreover $x_0^{2a-3} x_1^2 x_2 \in J_{2a}$, $x_0^{2a-2} x_2^2,x_0^{2a-4} x_1^4 \in \cn(J)_{2a}$,
and $(x_0^{2a-3} x_1^2 x_2)^2 = x_0^{2a-2} x_2^2 \cdot x_0^{2a-4} x_1^4$. \qedhere
\end{itemize}
\end{proof}

\begin{proposition}\label{hilb>=1}
Let $\preceq$ be any reverse term order and $p(z)$ an admissible polynomial of positive degree with Gotzmann number $r$.
\begin{itemize}
\item[(1)] If $p(r)\leq \binom{r-1+n}{n}$, then there is not the hilb-segment ideal for $p(z)$.
\item[(2)] If $p(z)=dz+1-g$, there exists the hilb-segment ideal $J$ for $p(z)$ if and only if
\begin{itemize} 
\item[(i)] $r=d$ or $r=d+1$, when $n=2$;
\item[(ii)] $r=d=1$, when $n>2$.
\end{itemize} 
\end{itemize} 
\end{proposition}

\begin{proof}
(1) By the hypothesis we have that $x_1^r$ belongs to the ideal, hence the Krull dimension must to be 1 by 
Proposition \ref{HS}, and we are done. 

(2) In this case, the hilb-segment ideal $J$ exists if and only if $p(r)=\binom{n+r-1}{n}+d$. 
Infact, the sous-\'escalier of $(J,x_0)_r$ contains the least $d$ terms not divided by $x_0$ and, since the term 
order is reverse and $r\geq d$, also the sous-\'escalier of $J_r$ must contain the same least $d$ terms not 
divided by $x_0$. Hence, by the Borel property, all the terms divided by $x_0$  must to belong to the 
sous-\'escalier of $J_r$, too. Thus, since $r=\binom{d}{2}+1-g$ by Example \ref{numeroGgrado1}, we get:
\[
dr+r-\binom{d}{2}=\binom{n+r-1}{n}+d \Leftrightarrow 
d=\frac{1}{2}\left( 2r-1\pm\sqrt{8 \binom{r+1}{2}-8\binom{r+n-1}{n}+1}\right)
\]
so that $J$ exists if and only if the argument $\Delta$ under the square root is not negative. 
By an easy calculation we obtain the thesis.
\end{proof}

\subsection{On gen-segment ideals for revlex order}

We describe some procedure to construct gen-segment ideals w.r.t revlex order with a given admissible 
polynomial $p(z)$. We have already observed that a hilb-segment ideal exists always and, thus, also a 
gen-segment ideal for a constant polynomial $p(z)=d$.

\begin{lemma}\label{j(n)}
If $p(z)=dz+1-g$ is an admissible polynomial with Gotzmann number $r$, there exist two integers 
$n\geq 2$ and $j(n)>0$ such that $\binom{j(n)-1+n}{n}\leq p(j(n)-1)$ and $p(j(n)+h)<\binom{j(n)+h+n}{n}$ 
for every $h\geq 0$.
\end{lemma}

\begin{proof}
Any projective scheme of dimension $1$ with Hilbert polynomial $p(z)$ has regularity $\leq r$, so that 
$p(r)<\binom{r+n}{n}$, for any $n\geq 2$. Now, it is enough to show that there exist integers $n\geq 2$ and 
$t<r$ such that $p(t)\geq \binom{t+n}{n}$. Since in the plane, i.e. for $n=2$, it holds 
$g\leq \frac{1}{2}(d-1)(d-2)$, then $p(t)=dt+1-g\geq dt+1-\frac{1}{2}(d-1)(d-2)$ and, for $t=d-1$, we have 
$d(d-1)+1-\frac{1}{2}(d-1)(d-2)=\binom{d-1+2}{2}$. Thus, $n=2$, $d\leq j(n)\leq r$.  
\end{proof}

\begin{proposition} \label{gen-segment}
Let $p(z)=dz+1-g$ be an admissible polynomial. For any $n\geq 2$ there exists a gen-segment ideal 
$I(n)\subset S$ w.r.t. revlex order with Hilbert polynomial $p(z)$.
\end{proposition}

\begin{proof} By Lemma \ref{j(n)} we can take an integer $n\geq 2$ for which there exists $j(n)>0$ such that 
$\binom{j(n)-1+n}{n}\leq p(j(n)-1)$ and $p(j(n)+h)<\binom{j(n)+h+n}{n}$ for every $h\geq 0$. First, we prove 
the thesis in this case.

Under the given assumptions we have $p(j(n))-\binom{j(n)-1+n}{n}\geq d=p(j(n))-p(j(n)-1)$, thus, by Remark 
\ref{espansione} (3), in $\Lambda_{p(j(n)),j(n)}$ there are at least $d$ terms $x^{\alpha}$ such that 
$\min(x^{\alpha})\geq1$ and we let $\tau_1<\cdots<\tau_d$ be the least  w.r.t. revlex order among them. We also 
set $N(t):=\mathbb T_t$, for every $0\leq t<j(n)$, $N(j(n)):=\Lambda_{p(j(n)),j(n)}$, $N(t):=x_0\cdot N_{t-1}\sqcup x_1^{h}\cdot\{\tau_1,\ldots,\tau_d\}$, 
for every $t=j(n)+h, h\geq 1$ and $N:=\sqcup_{t\geq 0} N(t)$. By construction $N\subset\mathbb T$ is such that  
$N_t=N(t), \forall\ t\geq0$ and $\mid N_t\mid=p(t), \forall\ t\geq j(n).$ Thus, the monomial ideal 
$I(n)\subset S$ such that $\cn(I(n))=N$ is, by construction, a gen-segment ideal with Hilbert polynomial $p(z)$. 
Moreover $G(I(n))_t=\emptyset$, for $t<j(n) \text{ and } t>j(n)+1$, so that $reg(I(n))\leq j(n)+1\leq r$.

Now, suppose that $n$ is such that $p(t)<\binom{t+n}{n}$ for every $t\geq 0$ and let $n_0:=\max\{n' \ \vert \ \exists j(n') : \binom{j(n')-1+n'}{n'}\leq p(j(n')-1) \text{ and } p(j(n'))<\binom{j(n')+n'}{n'}\}$. 
Above we proved that for such an $n_0$ there exists a gen-segment ideal $I(n_0)\subset K[x_0,\ldots,x_{n_0}]$ 
w.r.t. revlex order with Hilbert polynomial $p(z)$. Now, it is enough to observe that $I(n):=(I(n_0),x_{n_0+1},\ldots,x_n)\subset S$ 
is a gen-segment ideal  w.r.t. revlex order as claimed.
\end{proof}

\begin{remark}\label{rem:gen-segment}\rm 
Given an admissible polynomial $p(z)=dz+1-g$, if $n> 2$ is such that $p(t)<\binom{t+n}{n}$ for every $t\geq 0$ 
and \lq\lq there exists $l(n):=\min\{l\in \mathbb N : \sum_{i=1}^{n-1}\lambda_{i,l}(N_{l})\geq d\}$\rq\rq, 
by a similar procedure we can construct a gen-segment ideal $J(n)\subset S$ w.r.t. revlex order with Hilbert 
polynomial $p(z)$  different from those coming from the smaller $n'$'s as in the proof of Proposition 
\ref{gen-segment}. Indeed, under the given assumptions, $\Lambda_{p(l(n)),l(n)}\subset\mathbb T_{l(n)}$ does 
no longer contain at least $d$ terms $x^{\alpha}$ with $\min(x^{\alpha})\geq 1$, but surely its expansion in 
degree $l(n)+1$ does it and we let $\bar{\tau_1}<\cdots<\bar{\tau_d}$ be the least  w.r.t. revlex order of them. 
Similarly as before, we take $M(t):=\mathbb T_t$, for every $0\leq t<l(n)$, $M(l(n)):=\Lambda_{p(l(n)),l(n)}$, 
$M(l(n)+1):=x_0\cdot M(l(n))\sqcup \{\bar{\tau_1},\ldots,\bar{\tau_d}\}$, $M(t):=x_0\cdot M(t-1)\sqcup x_1^{t-l(n)-1}\{\bar{\tau_1},\ldots,\bar{\tau_d}\}$ 
for every $t>l(n)+1$. We finally let $J(n)$  be the gen-segment ideal such that $\cn(J(n))=M:=\sqcup_{t\geq 0} M(t)$ 
and note that it has $p(z)$ as Hilbert polynomial and regularity $\leq l(n)+2$.
\end{remark}

\begin{example}\rm 
(1) The Gotzmann number of the admissible polynomial $p(z)=6z-3$ is $12$ and we obtain the following 
gen-segment ideals:
\begin{itemize}
\item[(i)] if $n=2$, we can apply the procedure described in the proof of Proposition \ref{gen-segment} with 
$j(2)=9$ and construct the ideal $I(2)=(x_2^9,x_1 x_2^8,x_1^2 x_2^7,x_1^3 x_2^6)$; 
\item[(ii)] if $n=3$, there is not a $j(3)$, yet we can apply the procedure described in Remark 
\ref{rem:gen-segment} with $l(3)=2$, since $p(t)<\binom{3+t}{t}$, for every $t>0$, obtaining  
$J(3)=(x_3^2,x_2^2x_3,x_2^4)$ besides $(I(2),x_3)$;
\item[(iii)] if $n\geq 4$, neither  $j(n)$ nor $l(n)$ exist and we have only $(I(2),x_3,\ldots,x_n)$ and 
$(J(3),x_4,\ldots,$ $x_n)$.
\end{itemize}

(2) The Gotzmann number of the admissible polynomial $p(z)=7z+1$ is $22$ and we have 
$j(2)=12,\ j(3)=4,\ j(4)=3,\ j(5)= j(6)=j(7)=2$, while for $n\geq 8$, there is not a $j(n)$. Moreover, 
by definition of $j(n)$, we have $\alpha(n):=p(j(n))-\binom{j(n)-1+n}{n}\geq 7,$ thus we can apply the 
procedure described in the proof of Proposition \ref{gen-segment} and 
\begin{enumerate}[(i)]
\item for $n=2$, as $\alpha(2)=85-78=7$, we obtain $I(2)=(x_2^{12},x_1 x_2^{11},\ldots,x_1^5 x_2^7)$ 
($reg(I(2))=j(2)$);
\item for $n=3$, as $\alpha(3)=29-20=9>7$, we obtain $I(3)=(x_3^{4},x_2 x_3^{3},x_2^2 x_3^{2},x_2^3 x_3,x_2^4,$ $x_1 x_3^3,x_1^2 x_2 x_3^2,x_1^2 x_2^2 x_3)$ ($reg(I(3))=j(3)+1$), besides $(I(2),x_3)$;
\item for $n=4$, as $\alpha(4)=22-15=7$, we obtain $I(4)=((x_{4},x_3,x_2)^3,x_1(x_4,x_3)^2)$ ($reg(I(4))=j(4)$), besides $(I(2),x_3,x_4)$ and $(I(3),x_4)$;
\item for $n=5$, as $\alpha(5)=15-6=9>7$, we obtain $I(5)=((x_{5},x_4,x_3)^2,x_2^2 x_5, x_1x_2x_5,$  $x_2^2 x_4,x_1x_2x_4, x_2^2 x_3, x_2^3)$\hfill ($reg(I(5))=j(5)+1$),\hfill besides\hfill $(I(2),x_3,x_4,x_5)$,\hfill $(I(3),$ $x_4,x_5)$,\\ $(I(4),x_5)$;
\item for $n=6$, as $\alpha(6)=15-7=8>7$, we obtain $I(6)=((x_6,x_{5},x_4,x_3)^2,x_2x_6,x_2x_5,$ $x_2x_4,x_2^2x_3, x_1x_2x_3, x_2^3)$,  ($reg(I(6))=j(6)+1$), besides $(I(2),x_3,x_4,x_5,x_6)$, $(I(3),x_4,$ $x_5,x_6)$, $(I(4),x_5,x_6)$ and $(I(5),x_6)$;
\item for $n=7$, as $\alpha(7)=15-8=7$, we obtain $I(7)=(x_7,x_6,x_{5},x_4,x_3,x_2)^2$,  
($reg(I(7))=j(7)$), besides $(I(2),x_3,x_4,x_5,x_6,x_7)$, $(I(3),x_4,x_5,x_6,x_7)$, $(I(4),x_5,$ $x_6,x_7)$, $(I(5),x_6,x_7)$ and $(I(6),x_7)$;
\item for $n\geq 8$, $j(n)$ does not exist yet, as $\lambda_{1,1}(\Lambda_{p(1),1})\geq7$ we have $J(n)=(J(8),x_9,\ldots,$ $x_n)$ besides the ideals coming from $I(m), m\leq7$.
\end{enumerate}
\end{example}

\begin{proposition}
The saturated segment ideal $L(p(z))\subset S$ w.r.t. lex order with Hilbert polynomial $p(z)$ is a gen-segment 
ideal w.r.t. the revlex order if and only if $\deg(p(Z))\leq 1$ or there are only two generators of degree $>1$.
\end{proposition}

\begin{proof}
In section $1$ we have already recalled that, given an admissible polynomial $p(z)$ of degree $\ell$, there 
exist unique integers $m_0\geq m_1\geq \cdots\geq m_{\ell}\geq 0$ such that $p(z)=\sum_{i=0}^{\ell} \binom{z+i}{i+1}- \binom{z+i-m_i}{i+1}$ \cite{M,H66,B}. 
Let $a_{\ell}:=m_{\ell}, a_{\ell-1}:=m_{\ell-1}-m_{\ell},\ldots,a_0:=m_0-m_1$. Note that $L(p(z))\subset S$ has 
the $n+1-\ell-2$ greatest variables as generators of degree $1$, i.e. $\mathcal N(L(p(z)))_1= \{x_0,\ldots, x_{\ell+1}\}$. 
Thus, for every $j\leq a_{\ell}$, the greatest term of $\mathcal N(L(p(z)))_j$ is $x^j_{\ell+1}$, w.r.t. both 
lex and revlex orders (namely $\mathcal N(L(p(z)))_j=\mathbb T_j\cap K[x_0,\ldots,x_{\ell+1}]$). In degree 
$a_{\ell}+1$ the ideal $L(p(z))$ has a new generator $x_{\ell+1}^{a_{\ell}+1}$, so that $\mathcal N(L(p(z)))_{a_{\ell}+1}=(\mathbb T_{a_{\ell}+1}\cap K[x_0,\ldots,x_{\ell+1}]) \setminus \{x_{\ell+1}^{a_{\ell}+1}\}$, 
therefore its greatest term, w.r.t. both lex and revlex orders, is $x_{\ell}x_{\ell+1}^{a_{\ell}}$, and so on 
until there is a new generator in degree $a_{\ell}+a_{\ell-1}+1$ if $a_{l-2}\not=0$, which is $x_{\ell}^{a_{\ell-1}+1} x_{\ell+1}^{a_{\ell}}$ 
(or, if $a_{\ell-2}=0$, the new generator is $x_{\ell}^{a_{\ell-1}} x_{\ell+1}^{a_{\ell}}$). At this point, 
the greatest term in $\mathcal N(L(p(z)))_{a_{\ell}+a_{\ell-1}+1}$ is $x_{\ell}^{a_{\ell-1}+2} x_{\ell+1}^{a_{\ell}-1}$ 
w.r.t. revlex order and $x_{\ell-1}x_{\ell}^{a_{\ell-1}} x_{\ell+1}^{a_{\ell}}$ w.r.t. lex order (similarly for the 
case in parenthesis). Moreover, since the new generator of $L(p(z))$ at degree $a_{\ell}+a_{\ell-1}+a_{\ell-2}+1$ 
is $x_{\ell-1}^{a_{\ell-2}+1}x_{\ell}^{a_{\ell-1}} x_{\ell+1}^{a_{\ell}}$ (if $\ell=2$ the third generator of 
degree $>1$ is $x_1^{a_0} x_2^{a_1} x_3^{a_2}$ or, if $a_{\ell-3}=0$, is  $x_{\ell-1}^{a_{\ell-2}}x_{\ell}^{a_{\ell-1}} x_{\ell+1}^{a_{\ell}}$), 
it is not the greatest term w.r.t. revlex order.
\end{proof}

\begin{example}\label{lexsaturato}\rm 
(i) The ideal $L(p(z))=(x_4,x_3^5,x_2^3x_3^4,x_1^6 x_2^2 x_3^4)$ is the saturated segment ideal w.r.t. lex 
in $K[x_0,\ldots,x_4]$, with Hilbert polynomial $p(z)=2z^2+2z+1$ and Gotzmann number $12$, but is not a 
gen-segment ideal w.r.t. revlex order.

(ii) The ideal $L(p(z))=(x_5,x_4^5,x_3^2 x_4^4)$ is the saturated segment ideal w.r.t. lex in $K[x_0,\ldots,x_5]$, 
with Hilbert polynomial $p(z)=2/3z^3+2z^2-11/3z+10$ and Gotzmann number $6$, and is also a gen-segment ideal 
w.r.t. revlex order.
\end{example}


\section{Saturations of Borel ideals and Hilbert polynomial} 

Let $J\subset S$ be a Borel ideal. Recall that in our notation the (Borel) ideal $J^{sat}$ is obtained by setting $x_0=1$ in each minimal generator of $J$ (Proposition \ref{reeves}(i)). In this section we let $J_{x_0}:=J^{sat}$ and denote $J_{x_0x_1}$  the Borel ideal obtained by setting $x_0=x_1=1$ in the minimal generators of $J$. We call $J_{x_0x_1}$ the {\it $x_1$-saturation} of $J$ and say that $J$ is {\it $x_1$-saturated} if $J=J_{x_0x_1}$, so that an ideal $ J \, x_1$-saturated is also saturated.

\begin{remark}\label{polinomio}\rm
 An  ideal $J\subset S$, which is $x_1$-saturated and has Hilbert polynomial $p(z):=p_{{S}/{J}}(z)$, has the same minimal generators of the saturated Borel ideal $J\cap K[x_1,\ldots,x_n]\subset K[x_1,\ldots,x_n]$, for which the Hilbert polynomial is $\Delta p(z)$. 
\end{remark}

The following result is analogous to Theorem 3 of \cite{RA}, where the notion of \lq\lq fan\rq\rq\ is used. 
Here we apply the combinatorial properties of Borel ideals only.

\begin{proposition}\label{differenza}
Let $J\subset S$ be a saturated Borel ideal with Hilbert polynomial $p(z)$ and Gotzmann number $r$. 
Let $I=J_{x_0x_1}$ be the $x_1$-saturation of $J$ and let $q:=\dim_K I_r- \dim_K J_r$. Then
\begin{itemize}
\item[(i)] $p_{S/I}(z)=p(z)-q$;
\item[(ii)] $q$ is equal to the sum of the exponents of $x_1$ in the minimal generators of $J$.
\end{itemize}
\end{proposition}

\begin{proof}
(i) We show that if $q=\dim_K I_s-\dim_K J_s$ then $q=\dim_K I_{s+1}-\dim_K J_{s+1}$, for every $s\geq r$. 
Let $x^{\beta_1},\ldots,x^{\beta_q}$ be the terms of $I_s\setminus J_s$. Thus, $x_0x^{\beta_1},\ldots,x_0x^{\beta_q}$ 
are terms of $I_{s+1}\setminus J_{s+1}$ and so $\dim_K I_{s+1}-\dim_K J_{s+1}\geq q$, since $x_0x^{\beta_i}$ 
belongs to $J_{s+1}$ if and only if $x^{\beta_i}$ belongs to $J_{s}$ being $J$ saturated. Now, for obtaining the 
opposite inequality it is enough to show that every term of $I_{s+1}\setminus J_{s+1}$ is divided by $x_0$. 
Let $x^{\gamma}\in I_{s+1}$ be such that $\min (x^{\gamma})\geq1$ and let $x^{\alpha}$ be a minimal generator of $I$ such 
that $x^{\gamma}=x^{\alpha}x^{\delta}$. Since $J$ is saturated and $I$ is the $x_1$-saturation of $J$, 
$x^{\alpha}x_1^a$ is a minimal generator of $J$ for some non negative integer $a$. Hence, for every $x^{\delta'}$ of degree 
$s+1-\vert\alpha \vert$ and with $\min (x^{\delta'})\geq1$, by the Borel property 
$x^{\alpha}x^{\delta'}$ belongs to $J_s$.  In particular, $x^{\gamma}\in J_{s+1}$.

(ii) Let $x^{\alpha_1}x_1^{s_1}, \ldots, x^{\alpha_h}x_1^{s_h}$ be the minimal generators of $J$, with $x_1$ 
 not dividing $x^{\alpha_i}, \forall \, 1\leq i\leq h$. As the $\sum s_i$ terms $x^{\alpha_i}x_1^{s_i-t}x_0^{r-\vert\alpha_i\vert-s_i+t}, 1\leq t\leq s_i,$ 
are in $I_r\setminus J_r$, one has $q\geq \sum s_i$. Vice versa, we show that each term $x^{\delta}$ in $I_r\setminus J_r$ is of the previous type. We can write $x^{\delta}=x^{\beta} x_0^{r-\vert\beta\vert-u} x_1^u$, with $\min(x^{\beta})\geq2$ 
 and $u<s_i$. Let $s$ be the minimum non negative integer such that $x^{\beta} x_1^s$ is in $J$. 
Then there exists $i$ such that $x^{\alpha_i}x_1^{s_i} \vert x^{\beta}x_1^s$, i.e. $x^{\alpha_i} \vert x^{\beta}$ 
and $s_i\leq s$. By the definition of $s$, we get $s_i=s$ and there exists $x^{\gamma}$ with  $\min(x^{\gamma})\geq2$ 
such that $x^{\beta}=x^{\alpha_i}x^{\gamma}$. Since $x^{\beta}$ does not belong to $J$ we have $\vert\gamma\vert<s_i=s$,
 otherwise $x^{\alpha_i}x_1^{\vert\gamma\vert}$ and hence, 
by the Borel property, $x^{\beta}=x^{\alpha_i}x^{\gamma}$ should belong to $J$. Now we can take $x^{\beta}x_1^{s-\vert\gamma\vert}$ and observe that this term belongs to 
$J$ because it follows $x^{\alpha_i}x_1^s$ in the Borel relation. Thus $s\leq s-\vert \gamma \vert$, so that 
$\gamma=0$, i.e. $x^{\beta}=x^{\alpha_i}$ as claimed.
\end{proof}

\begin{proposition}\label{togliendo}
Let $J\subset S$ be a saturated Borel ideal with Hilbert polynomial $p(z)$ and Gotzmann number $r$. 
Let $x^{\beta}x_0$ be a term of $J$ of degree $s\geq r$ which is minimal in $J$ w.r.t. $<_B$. Then the ideal 
$I:=(G((J_s))\setminus \{x^{\beta}x_0\})$ is Borel and $p_{S/I}(z)=p(z)+1$.
\end{proposition}

\begin{proof}
First, note that $I_s$ is closed w.r.t. $<_B$ by Remark \ref{cicca}. We show that, for every $t\geq 0$, 
$x^{\beta}x_0^{1+t}$ is the unique term in $J_{s+t}\setminus I_{s+t}$. For $t=0$ we have the hypothesis. 
For $t>0$, note that $x^{\beta}x_0^{1+t}$ cannot belong to $I$. On the contrary, there would be a term 
$x^{\gamma}\in I_s$ such that $x^{\gamma}\not=x^{\beta}x_0$ and $x^{\gamma}\mid x^{\beta}x_0^{1+t}$. But 
every degree $s$ factor  of $x^{\beta}x_0^{1+t}$ different from $x^{\beta}x_0$ is lower w.r.t. $<_B$ and 
so it cannot belong to $I_s$. Then, $x^{\beta}x_0^{1+t}\notin I_{s+t}$. If $x^{\alpha}$ is a 
term of $J_{s+t}\setminus I_{s+t}$, there exists a term of $J_{s+t-1}\setminus I_{s+t-1}$ which divides 
$x^{\alpha}$. By induction, this term is $x^{\beta}x_0^{t}$ and the thesis follows from the fact that every 
multiple of degree $s+t$ of $x^{\beta}x_0^{t}$, different from $x^{\beta}x_0^{1+t}$, belongs to $I_{s+t}$. 
\end{proof}

\begin{proposition}\label{$x_1$-saturation}
Let $I$ and $J$ be Borel ideals of $S$. If for every $s\gg 0$ we have $I_s \subset J_s$ and 
$p_{S/I}(z)=p_{S/J}(z)+a$, with $a\in \mathbb N$, then $I$ and $J$ have the same $x_1$-saturation. 
\end{proposition}

\begin{proof}
Let $s\geq \max\{reg(I),reg(J)\}$. In case $a=1$, there exists a unique term in $J_{s+t}\setminus I_{s+t}$, 
for every $t\geq 0$. Let $x^{\alpha}$ be the unique term in $J_s\setminus I_s$. Then, both $x^{\alpha}x_0$ and 
$x^{\alpha}x_1$ belong to $J_{s+1}$. By the Borel property, $x^{\alpha}x_1$ must be in 
$I_{s+1}$ and so the unique term in $J_{s+t}\setminus I_{s+t}$ is $x^{\alpha}x_1^t$. This is enough to say that $I$ and $J$ have the same $x_1$-saturation. If $a>1$, the thesis follows 
by induction applying Proposition \ref{togliendo}.
\end{proof}

\begin{corollary}\label{polinomio minimo}
Let $p(z)$ be an admissible polynomial of degree $h\leq n$ and 
$P:=\{q(z)=p(z)+u \ \vert \ u\in \mathbb Z \text{ and } q(z)  \text{admissible} \}$ 
the set of all admissible polynomials of degree $h$ which differ from $p(z)$ only for an integer. Then
\begin{itemize}
\item[(i)] there is a polynomial $\hat p(z)$ in $P$ such that, for every $q(z)$ in $P$, $q(z)=\hat p(z)+c$ with $c\geq 0$;
\item[(ii)] every saturated Borel ideal $I$ with Hilbert polynomial $p_{S/I}=\hat p(z)$ is $x_1$-saturated.
\end{itemize}
\end{corollary}

\begin{proof}
(i) Every admissible polynomial $p(z)$ has a unique saturated lex segment ideal $L(p(z))$. If $H$ is the 
saturated lex segment ideal of $p(z)+u$, then we have $H\subset L(p(z))$ if $u>0$ and 
$L(p(z))\subset H$ if $u<0$. Thus, we can apply Proposition \ref{$x_1$-saturation}, obtaining that $L(p(z))$ and $H$ have 
the same $x_1$-saturation $I$. We claim that $\hat p(z)$ is the Hilbert polynomial of $I$. Indeed, by 
Proposition \ref{differenza} the Hilbert polynomial of $I$ is of type $p(z)-q$. If $\hat p(z)=p(z)-q-t$ with 
$t\geq 0$, then the saturated lex segment ideal of $\hat p(z)$ should have $I$ as $x_1$-saturation 
and should be contained in $I$, that is possible only if $t=0$.

(ii) Let $J$ be a Borel ideal with $\hat p(z)$ as Hilbert polynomial. If $J$ were not $x_1$-saturated, 
by Proposition \ref{differenza} the $x_1$-saturation of $J$ should have a Hilbert polynomial of type 
$\hat p(z)-q$, with $q>0$, against the definition of $\hat p(z)$,
\end{proof}

\begin{definition}\label{def:minimo}\rm
The polynomial $\hat p(z)$ of Corollary \ref{polinomio minimo} is called {\em minimal polynomial}.
\end{definition}

\begin{remark}\rm 
An alternative proof of the previous statement can be obtained by following the construction of the Gotzmann 
number. 
\end{remark}

\begin{example}\rm
By Proposition \ref{differenza}, a Borel ideal with a minimal Hilbert polynomial is $x_1$-saturated. The vice 
versa is not true. For example, the ideal $I=(x_3^2,x_2x_3,x_2^2)\subset K[x_0,x_1,x_2,x_3]$ is $x_1$-saturated 
and is a reg-segment ideal w.r.t. the revlex order. The corresponding Hilbert polynomial is $p_{S/I}(z)=3z+1$ 
which is not minimal because the Borel ideal $(x_3,x_2^3)$ has Hilbert polynomial $3z$. 
\end{example}

\begin{remark}\rm
From the proof of Corollary \ref{polinomio minimo} we deduce the following fact. Let 
$I\subset K[x_0,\ldots,$ $x_n]$ be a Borel ideal with Hilbert polynomial $p(z)$. If 
$I=I_{x_1}\cdot K[x_0,\ldots,x_n]$, where $I_{x_1}\subset K[x_1,\ldots,x_n]$ is the segment ideal w.r.t. 
lex order with Hilbert polynomial $\Delta p(z)$, then $p(z)=\hat p(z)$.
\end{remark}


\section{An algorithm to compute saturated Borel ideals}

In this section, by exploiting the arguments of section 4, we describe an algorithm for computing all the saturated Borel ideals with a given Hilbert 
polynomial $p(z)$.  We first give an efficient strategy to find the minimal 
elements in a Borel set $B$, that consists in representing $B$ by a connected planar graph, in which the nodes 
are the terms of $B$ and the edges are the elementary moves which connect the terms. In Figure \ref{esBorel} we 
give some examples showing that it is easy to single out the minimal terms looking at these graphs. 

\begin{figure}
\begin{center} 
 \begin{tikzpicture}[>=latex',line join=bevel,scale=0.8]
\tikzstyle{ideal}=[draw,rectangle,minimum width=1.1cm,minimum height=.55cm]
\tikzstyle{quotient}=[]
\node (N_14) at (248bp,124bp) [quotient] {\small $x_{0}x_{1}x_{2}$};
  \node (P_5) at (424bp,238bp) [ideal] {\small $x_{1}x_{2}x_{3}$};
  \node (P_2) at (399bp,276bp) [quotient] {\small $x_{2}^{2}x_{3}$};
  \node (P_15) at (397bp,86bp) [quotient] {\small $x_{0}x_{1}^{2}$};
  \node (N_10) at (302bp,238bp) [ideal] {\small $x_{0}x_{3}^{2}$};
  \node (P_6) at (375bp,200bp) [quotient] {\small $x_{1}x_{2}^{2}$};
  \node (P_4) at (448bp,276bp) [ideal] {\small $x_{1}x_{3}^{2}$};
  \node (N_11) at (305bp,200bp) [ideal,fill=black!10] {\small $x_{0}x_{2}x_{3}$};
  \node (N_13) at (309bp,162bp) [quotient] {\small $x_{0}x_{1}x_{3}$};
  \node (P_8) at (370bp,162bp) [quotient] {\small $x_{1}^{2}x_{2}$};
  \node (P_3) at (376bp,238bp) [quotient] {\small $x_{2}^{3}$};
  \node (P_7) at (424bp,200bp) [ideal] {\small $x_{1}^{2}x_{3}$};
  \node (P_11) at (480bp,200bp) [quotient] {\small $x_{0}x_{2}x_{3}$};
  \node (P_9) at (371bp,124bp) [ideal] {\small $x_{1}^{3}$};
  \node (P_10) at (478bp,238bp) [ideal] {\small $x_{0}x_{3}^{2}$};
  \node (N_12) at (248bp,162bp) [quotient] {\small $x_{0} x_{2}^{2}$};
  \node (P_12) at (424bp,162bp) [quotient] {\small $x_{0}x_{2}^{2}$};
  \node (P_13) at (485bp,162bp) [ideal] {\small $x_{0}x_{1}x_{3}$};
  \node (P_14) at (424bp,124bp) [quotient] {\small $x_{0}x_{1}x_{2}$};
  \node (N_18) at (248bp,48bp) [quotient] {\small $x_{0}^{2}x_{1}$};
  \node (P_16) at (485bp,124bp) [quotient] {\small $x_{0}^{2}x_{3}$};
  \node (P_17) at (451bp,86bp) [quotient] {\small $x_{0}^{2}x_{2}$};
  \node (P_18) at (424bp,48bp) [quotient] {\small $x_{0}^{2}x_{1}$};
  \node (P_19) at (424bp,11bp) [quotient] {\small $x_{0}^{3}$};
  \node (M_18) at (72bp,48bp)  {\small $x_{0}^{2}x_{1}$};
  \node (N_16) at (309bp,124bp) [quotient] {\small $x_{0}^{2}x_{3}$};
  \node (M_9) at (18bp,124bp)  {\small $x_{1}^{3}$};
  \node (M_8) at (18bp,162bp)  {\small $x_{1}^{2}x_{2}$};
  \node (N_17) at (275bp,86bp) [quotient] {\small $x_{0}^{2} x_{2}$};
  \node (M_3) at (18bp,238bp)  {\small $x_{2}^{3}$};
  \node (M_2) at (45bp,276bp)  {\small $x_{2}^{2}x_{3}$};
  \node (M_1) at (72bp,314bp)  {\small $x_{2}x_{3}^{2}$};
  \node (M_0) at (72bp,351bp)  {\small $x_{3}^{3}$};
  \node (M_7) at (72bp,200bp)  {\small $x_{1}^{2}x_{3}$};
  \node (M_6) at (18bp,200bp)  {\small $x_{1}x_{2}^{2}$};
  \node (M_5) at (72bp,238bp)  {\small $x_{1}x_{2}x_{3}$};
  \node (M_4) at (99bp,276bp)  {\small $x_{1}x_{3}^{2}$};
  \node (N_15) at (221bp,86bp) [quotient] {\small $x_{0}x_{1}^{2}$};
  \node (N_8) at (194bp,162bp) [quotient] {\small $x_{1}^{2}x_{2}$};
  \node (N_9) at (194bp,124bp) [quotient] {\small $x_{1}^{3}$};
  \node (P_1) at (423bp,314bp) [ideal] {\small $x_{2}x_{3}^{2}$};
  \node (N_4) at (274bp,276bp) [ideal] {\small $x_{1}x_{3}^{2}$};
  \node (N_5) at (252bp,238bp) [ideal] {\small $x_{1}x_{2}x_{3}$};
  \node (N_6) at (198bp,200bp) [quotient] {\small $x_{1}x_{2}^{2}$};
  \node (N_7) at (252bp,200bp) [quotient] {\small $x_{1}^{2}x_{3}$};
  \node (N_0) at (252bp,351bp) [ideal] {\small $x_{3}^{3}$};
  \node (N_1) at (252bp,314bp) [ideal] {\small $x_{2} x_{3}^{2}$};
  \node (N_2) at (230bp,276bp) [ideal] {\small $x_{2}^{2} x_{3}$};
  \node (N_3) at (203bp,238bp) [ideal,fill=black!10] {\small $x_{2}^{3}$};
  \node (P_0) at (423bp,351bp) [ideal] {\small $x_{3}^{3}$};
  \node (M_19) at (72bp,11bp)  {\small $x_{0}^{3}$};
  \node (N_19) at (248bp,11bp) [quotient] {\small $x_{0}^{3}$};
  \node (M_17) at (99bp,86bp)  {\small $x_{0}^{2} x_{2}$};
  \node (M_16) at (133bp,124bp)  {\small $x_{0}^{2} x_{3}$};
  \node (M_15) at (45bp,86bp)  {\small $x_{0} x_{1}^{2}$};
  \node (M_14) at (72bp,124bp)  {\small $x_{0}x_{1}x_{2}$};
  \node (M_13) at (133bp,162bp)  {\small $x_{0}x_{1}x_{3}$};
  \node (M_12) at (72bp,162bp)  {\small $x_{0}x_{2}^{2}$};
  \node (M_11) at (133bp,200bp)  {\small $x_{0}x_{2}x_{3}$};
  \node (M_10) at (133bp,238bp)  {\small $x_{0} x_{3}^{2}$};
  \draw [->] (P_9) -- (P_15);
  \draw [->] (N_3) -- (N_6);
  \draw [->] (P_18) -- (P_19);
  \draw [->] (P_11) -- (P_13);
  \draw [->] (N_2) -- (N_3);
  \draw [->] (N_12) -- (N_14);
  \draw [->] (M_1) -- (M_2);
  \draw [->] (M_2) -- (M_5);
  \draw [->] (M_16) -- (M_17);
  \draw [->] (N_9) -- (N_15);
  \draw [->] (P_3) -- (P_6);
  \draw [->] (P_2) -- (P_3);
  \draw [->] (N_4) -- (N_10);
  \draw [->] (P_6) -- (P_8);
  \draw [->] (N_6) -- (N_12);
  \draw [->] (M_1) -- (M_4);
  \draw [->] (P_7) -- (P_8);
  \draw [->] (M_6) -- (M_8);
  \draw [->] (M_13) -- (M_14);
  \draw [->] (M_8) -- (M_14);
  \draw [->] (M_11) -- (M_12);
  \draw [->] (N_4) -- (N_5);
  \draw [->] (P_11) -- (P_12);
  \draw [->] (M_8) -- (M_9);
  \draw [->] (P_16) -- (P_17);
  \draw [->] (N_0) -- (N_1);
  \draw [->] (M_5) -- (M_7);
  \draw [->] (N_5) -- (N_7);
  \draw [->] (N_15) -- (N_18);
  \draw [->] (P_13) -- (P_16);
  \draw [->] (M_13) -- (M_16);
  \draw [->] (M_14) -- (M_15);
  \draw [->] (M_6) -- (M_12);
  \draw [->] (M_15) -- (M_18);
  \draw [->] (M_7) -- (M_13);
  \draw [->] (N_14) -- (N_15);
  \draw [->] (N_2) -- (N_5);
  \draw [->] (M_12) -- (M_14);
  \draw [->] (N_13) -- (N_14);
  \draw [->] (M_2) -- (M_3);
  \draw [->] (P_5) -- (P_7);
  \draw [->] (P_8) -- (P_14);
  \draw [->] (M_17) -- (M_18);
  \draw [->] (P_7) -- (P_13);
  \draw [->] (P_14) -- (P_17);
  \draw [->] (N_11) -- (N_13);
  \draw [->] (N_8) -- (N_9);
  \draw [->] (M_18) -- (M_19);
  \draw [->] (M_5) -- (M_11);
  \draw [->] (N_1) -- (N_4);
  \draw [->] (P_8) -- (P_9);
  \draw [->] (M_10) -- (M_11);
  \draw [->] (P_15) -- (P_18);
  \draw [->] (N_16) -- (N_17);
  \draw [->] (P_5) -- (P_6);
  \draw [->] (N_7) -- (N_8);
  \draw [->] (M_5) -- (M_6);
  \draw [->] (N_5) -- (N_6);
  \draw [->] (N_10) -- (N_11);
  \draw [->] (M_4) -- (M_10);
  \draw [->] (N_14) -- (N_17);
  \draw [->] (P_5) -- (P_11);
  \draw [->] (M_3) -- (M_6);
  \draw [->] (N_1) -- (N_2);
  \draw [->] (P_4) -- (P_10);
  \draw [->] (N_7) -- (N_13);
  \draw [->] (P_6) -- (P_12);
  \draw [->] (N_17) -- (N_18);
  \draw [->] (P_17) -- (P_18);
  \draw [->] (M_0) -- (M_1);
  \draw [->] (N_18) -- (N_19);
  \draw [->] (M_14) -- (M_17);
  \draw [->] (P_13) -- (P_14);
  \draw [->] (M_4) -- (M_5);
  \draw [->] (M_7) -- (M_8);
  \draw [->] (N_11) -- (N_12);
  \draw [->] (N_8) -- (N_14);
  \draw [->] (N_13) --  (N_16);
  \draw [->] (P_1) --  (P_2);
  \draw [->] (N_6) -- (N_8);
  \draw [->] (P_14) -- (P_15);
  \draw [->] (P_4) -- (P_5);
  \draw [->] (P_1) -- (P_4);
  \draw [->] (N_5) -- (N_11);
  \draw [->] (P_10) -- (P_11);
  \draw [->] (M_9) -- (M_15);
  \draw [->] (M_11) -- (M_13);
  \draw [->] (P_2) -- (P_5);
  \draw [->] (P_0) -- (P_1);
  \draw [->] (P_12) -- (P_14);
\end{tikzpicture}
\caption{\label{esBorel} Here are: on the left the graph of $K[x_0,\ldots,x_3]_3$, in the center the graph of $(x_3^2,x_2x_3,x_2^3)_3$, where we coloured the minimal elements, on the right the graph of $(x_3^2,x_1x_3,x_1^3)_3$ which is not Borel (the terms in the ideal are the boxed ones).} 
\end{center}
\end{figure}

Let $0\leq k<n$ be an integer. Recall that, if $I\subset K[x_k,\ldots,x_n]$ is a saturated Borel ideal 
which has a non null Hilbert polynomial $p(z)$ with Gotzmann number $r$, then $J:= \frac{(I,x_k)}{(x_k)}\subset K[x_{k+1},\ldots,x_n]$ 
has Hilbert polynomial $\Delta p(z)$, being $x_k$ a non zero-divisor on $\frac{K[x_k,\ldots,x_n]}{I}$. 

This fact tells that every saturated Borel ideal $I\subset K[x_k,\ldots,x_n]$ with Hilbert polynomial $p(z)$ 
\lq\lq comes from\rq\rq\ a Borel ideal $J\subset K[x_{k+1},\ldots,x_n]$ with Hilbert polynomial $\Delta p(z)$ 
and generated in degrees $\leq r$. So, our idea to construct all saturated Borel ideal with given Hilbert 
polynomial $p(z)$ consists in applying a recursion on the number of variables: in the hypothesis of knowing all 
Borel ideals $J$ in $n-k$ variables generated in degrees $\leq r$ with Hilbert polynomial $\Delta p(z)$, we 
construct the saturated Borel ideals $I$ in $n-k+1$ variables such that $J:= \frac{(I,x_k)}{(x_k)}$ for some of 
the ideals $J$.

Let $J\subset K[x_{k+1},\ldots,x_n]$ be a Borel ideal with Hilbert polynomial $\Delta p(z)$ and 
$\overline{I}:=$ $(J^{sat}\cdot K[x_k,\ldots,x_n])_r$, where $r$ is the Gotzmann number of $p(z)$. Let 
$\overline{N}$ be the set of terms $x^{\alpha}$ of $K[x_k,\ldots,x_n]_r$ such that there exists a composition 
$F$ of elementary moves of type $\down_j$ and a term $\tau$ of $\mathcal N(J)_r$ such that $F(\tau)=x^{\alpha}$. 
Hence, by construction, the terms of $\overline{N} \setminus \mathcal N(J)$ are not maximal and $\overline N$ is 
contained in the \emph{sous-\'escalier} of any ideal of $K[x_k,\ldots,x_n]$ having $J$ as hyperplane section. 
Note that the Gotzmann number of $\Delta^{k+1} p(z)$ is not higher than the Gotzmann number of $\Delta^k p(z)$.

\begin{lemma}\label{rem:equivalenza}
$\mathcal N(\overline I)_r=\overline N$. 
\end{lemma} 

\begin{proof}
It is enough to show that $K[x_k,\ldots,x_n]_r=(\overline I,\overline N)$ (Figure \ref{partizionamento}). 
Indeed, let $x^\gamma=x_k^{\gamma_k} \cdots x_n^{\gamma_n}$ be in $K[x_k,\ldots,x_n]_r$. The term 
$x^\beta=(\up_k)^{\gamma_k}(x^\gamma)$ belongs to $K[x_{k+1},\ldots,x_n]_r$, hence is in $J_r$ or in 
$\mathcal N(J)_r$. If $x^\beta$ is in $J_r$, then $x_{k+2}^{\gamma_{k+2}}\cdots x_n^{\gamma_n}$ belongs to 
$J^{sat}_r$, hence to $\overline I$, otherwise $x^\gamma=(\down_{k+1})^{\gamma_k}(x^\beta)$ is in $\overline{N}$.
\end{proof}

\begin{figure}[!h]
\begin{center}
\setlength{\unitlength}{1cm}
 \begin{picture}(10,5)(-5,-2.5)
\thicklines
\qbezier(-1.25, 1.5)(-1.25, 1.6657)(-1.7626, 1.7828)
\qbezier(-1.7626, 1.7828)(-2.2751, 1.9)(-3.0, 1.9)
\qbezier(-3.0, 1.9)(-3.7249, 1.9)(-4.2374, 1.7828)
\qbezier(-4.2374, 1.7828)(-4.75, 1.6657)(-4.75, 1.5)
\qbezier(-4.75, 1.5)(-4.75, 1.3343)(-4.2374, 1.2172)
\qbezier(-4.2374, 1.2172)(-3.7249, 1.1)(-3.0, 1.1)
\qbezier(-3.0, 1.1)(-2.2751, 1.1)(-1.7626, 1.2172)
\qbezier(-1.7626, 1.2172)(-1.25, 1.3343)(-1.25, 1.5)

\qbezier(4.75, 1.5)(4.75, 1.6657)(4.2374, 1.7828)
\qbezier(4.2374, 1.7828)(3.7249, 1.9)(3.0, 1.9)
\qbezier(3.0, 1.9)(2.2751, 1.9)(1.7626, 1.7828)
\qbezier(1.7626, 1.7828)(1.25, 1.6657)(1.25, 1.5)
\qbezier(1.25, 1.5)(1.25, 1.3343)(1.7626, 1.2172)
\qbezier(1.7626, 1.2172)(2.2751, 1.1)(3.0, 1.1)
\qbezier(3.0, 1.1)(3.7249, 1.1)(4.2374, 1.2172)
\qbezier(4.2374, 1.2172)(4.75, 1.3343)(4.75, 1.5)

\thinlines
\put(-1.25,1.5){\line(0,-1){3.5}}
\put(4.75,1.5){\line(0,-1){3.5}}
\put(-4.75,1.5){\line(1,-1){3.5}}
\put(1.25,1.5){\line(1,-1){3.5}}

\put(-4.1,2.1){$K[x_{k+1},\ldots,x_n]$}
\put(-5,-0.65){$K[x_k,\ldots,x_n]$}
\put(-2,-0.55){\circle*{0.1}}
\put(-2,-0.5){\vector(0,1){.4}}
\put(-2,-0.1){\vector(0,1){.4}}
\put(-2,.4){\line(0,1){.2}}
\put(-2,.7){\line(0,1){.2}}
\put(-2,1){\vector(0,1){.4}}
\put(-2,1.4){\circle*{0.1}}
\put(-3,0.5){$(\up_k)^{\gamma_k}$}

\put(3.2,1.85){\line(1,-2){0.35}}

\qbezier(3.2, 0.0)(3.2, -0.2485)(3.2513, -0.4243)
\qbezier(3.2513, -0.4243)(3.3025, -0.6)(3.375, -0.6)
\qbezier(3.375, -0.6)(3.4475, -0.6)(3.4987, -0.4243)
\qbezier(3.4987, -0.4243)(3.55, -0.2485)(3.55, 0.0)

\put(3.2,0){\line(0,1){1.85}}
\put(3.55,0){\line(0,1){1.15}}

\put(4,1.6){\line(2,1){1}}
\put(4.3,0){\line(2,-1){1}}
\put(2.5,1.6){\line(-2,1){1}}
\put(2.6,0.5){\line(-2,-1){1}}

\put(5.1,2){$\mathcal N(J)_r$}
\put(5.4,-0.7){$\overline{N}$}
\put(1.25,2){$J$}
\put(1.3,-0.15){$\overline{I}$}
\end{picture}
\end{center}
\caption{\label{partizionamento} Partition of $K[x_k,\ldots,x_n]$.}
\end{figure}

\begin{proposition}\label{togliere}
With the above notation, the Hilbert polynomial $\overline{p}(z)$ for $\overline{I}$ differs from $p(z)$ only for a constant. If 
$\overline{q}=p(r)-\overline{p}(r)> 0$, execute the following instruction $\overline{q}$ times: 
select a minimal term $\tau$ in $\overline{I}_r$ and set $\overline{I}:=( G((\bar{I}_r))\setminus \{\tau\})$. 
After these $\overline{q}$ steps, the new ideal obtained has Hilbert polynomial $p(z)$.
\end{proposition}

\begin{proof}
The theses follow from the results presented in section 4.
\end{proof}

Proposition \ref{togliere} suggests the design of the following two routines \textsc{BorelGenerator} and 
\textsc{Remove}, that have been implemented by the second author in a software with an applet available at 
\url{http://www.dm.unito.it/dottorato/dottorandi/lella/borelEN.html}. 

\begin{table*}[!h]
\begin{algorithmic}
\Procedure{BorelGenerator}{$n$,$p(z)$,$r$,$k$} $\rightarrow \mathcal F$
\If{$p(z) = 0$} 
\State \Return $\big\{(1)\big\}$;
\Else
\State $\mathcal E \leftarrow$ \Call{BorelGenerator}{$n$,$\Delta p(z)$,$r$,$k+1$}; 
\State $\mathcal F \leftarrow \emptyset$;
\ForAll{$J \in \mathcal E$}
\State $\bar I \leftarrow J \cdot k[x_k,\ldots,x_n]$;
\State $q \leftarrow p(r) - \dim_k k[x_k,\ldots,x_n]_r  + \dim_k \bar I_r$;
\If{$q \geqslant 0$}
\State $\mathcal F \leftarrow \mathcal F\ \cup$ \Call{Remove}{$n$,$k$,$r$,$\bar I$,$q$};
\EndIf
\EndFor
\State \Return $\mathcal F$;
\EndIf
\EndProcedure
\end{algorithmic}

\medskip

\begin{algorithmic}
\Procedure{Remove}{$n$,$k$,$r$,$\bar I$,$q$} $\rightarrow \mathcal E$
\State $\mathcal E \leftarrow \emptyset$;
\If{$q = 0$}
\State \Return $\mathcal E \cup \bar I^\sat$;
\Else
\State $\mathcal F \leftarrow$ \Call{MinimalElements}{$\bar I$,$r$}
\ForAll{$x^\alpha \in \mathcal F$}
 \State $\mathcal E \leftarrow \mathcal E\ \cup$ \Call{Remove}{$n$,$k$,$r$,$(G((\bar I_r))\setminus x^\alpha)$,$q-1$};
\EndFor
\State \Return $\mathcal E$;
\EndIf
\EndProcedure
\end{algorithmic}
\end{table*}

\begin{remark}\rm\label{q}
The terms removed by our strategy are minimal in $\overline I$. An alternative strategy could consists in adding 
to $J_r K[x_k,\ldots,x_n]$ maximal terms of $\overline I_r\setminus J$. In this case, since we want that 
$\dim_K I_r = \binom{n-k+r}{r} - p(r)$ and we have already $\binom{n-(k+1)+r}{r} - \Delta p(r)$ terms of $J$, 
we should add 
\[
 q' = \binom{n-k+r}{r} - p(r) - \binom{n-(k+1)+r}{r} + \Delta p(r) = \binom{n-k-1+r}{r-1} - p(r-1) 
\]
terms for any $J$, where $q'$ depends only on $r$, $n-k$ and $p(z)$; hence, we will write $q'(r,n-k,p(z))$ 
instead of $q'$. On the other hand, the value of $\overline q = p(r) - \vert\overline N_r\vert= p(0)-\overline p(0)$ 
depends on $J$. Note that $q'+\overline q=\dim_K \overline I_r - \dim_K J_r$. Anyway, we observe that if 
$n-k> \deg(p(z))+1$ then $q'\geq \overline q$. The minimal polynomial $\hat p(z)$ of Definition \ref{def:minimo} 
can be recovered from $\Delta p(z)$ by the decomposition of Gotzmann in the following way. If
\[
\Delta p(z)= \binom{z+b_1}{b_1}+\binom{z+b_2-1}{b_2}+\ldots+\binom{z+b_t-(t-1)}{b_t}
\]
with $b_1\geq b_2\geq\ldots\geq b_t\geq 0$, then
\[
\hat p(z)= \binom{z+a_1}{a_1}+\binom{z+a_2-1}{a_2}+\ldots+\binom{z+a_t-(t-1)}{a_t}
\]
where $a_i=b_i+1$. The Gotzmann number of $\Delta p(z)$ is also the Gotzmann number $\hat r$ of $\hat p(z)$. 
If $r$ is the Gotzmann number of $p(z)$, then $r-\hat r=p(0)-\hat p(0)\geq p(0)-\overline p(0)=\overline q$. 
We prove that $q'\geq\overline q$ by induction on $c=r-\hat r$. If $c=0$ then we get $\overline q=0$. If $c>0$, 
by induction we have that $q'(r-1,n-k,p(z)-1)\geq \overline q-1$, hence
\[
\begin{split}
q'(r,n-k,p(z))&=\binom{r-1+n-k}{n-k}-p(r-1) = \\
& = \binom{r-2+n-k}{n-k}+\binom{r-1+n-k-1}{n-k-1}-p(r-1)+ \\
& \quad + p(r-2)-p(r-2)= \\ 
&=q'(r-1,n-k,p(z)-1)+\binom{r-1+n-k-1}{n-k-1}-\Delta p(r-1)-1 \geq \\
&\geq\overline q+\binom{r-1+n-k-1}{n-k-1}-\Delta p(r-1)-2
\end{split}
\]
and $\binom{r-1+n-k-1}{n-k-1}-\Delta p(r-1)\geq 2$ since $r-1$ is a bound from above of the Gotzmann number of 
$\Delta p(z)$ and $J$ is not a hypersurface because $n-k-1>deg(\Delta p(z))+1$.
\end{remark}

\begin{example}\rm
(a) If $p(z)=d$, then $r=d$ and $\hat r=0$, so $\overline q=d$ and $q'=\binom{d-1+n}{n}-d$. If moreover 
$n=\deg(p(z))+1$, then $q'=0$. 

(b) The Gotzmann number of $p(z)=3z+1$ is $r=4$ and, if $n=3$ and $k=0$, then $q'(r,n,p(z))=\binom{r-1+n}{n}-p(r-1)=20-10=10$ 
and $r-\hat r=1$. If $J_4=(x_3,x_2^3)_4$ we get $\vert\overline N_r \vert=12$ and $\overline q=1$, meanwhile if 
$J_4=(x_3^2,x_2x_3,x_2^2)$ we obtain $\vert\overline N_r \vert=13$ and $\overline q=0$. 
\end{example}


\section{Reverse Lexicographic point}

In this section exploiting results of \cite{LR} we study the points corresponding to hilb-segment ideals in 
the Hilbert scheme $\hilbd$ of subschems of $\mathbb P^n$ with Hilbert polynomial $p(z)=d$, where $d$ is a 
fixed positive integer. Recall that for $p(z)=d$ the Gotzmann number is $d$ itself.

\emph{From now}, $J\subset S$ is a hilb-segment ideal with respect to some term order $\preceq$ and with 
Hilbert polynomial $p(z)=d$ and let $\mathcal B:=\{x^{\beta} \in \mathcal N(J)_d \ : \ x_1x^{\beta} \in J \}$. 
Recall that $G(I)$ denotes the set of minimal generators of $I$ and $ed(St_h(J,\preceq))$ is the embedding 
dimension of the Gr\"obner stratum $St_h(J,\preceq)$.

\begin{lemma}\label{lemma d}
With the notation above, we obtain that $ed(St_h(J,\preceq)) \geq \vert G(J)\vert \cdot \vert \mathcal B \vert$.
\end{lemma}

\begin{proof} 
With the same notation introduced in section 2, by Corollary 4.8(i) of \cite{LR} it is enough to look at the 
variables $c_{\alpha \beta}$ appearing in the polynomials $F_{\alpha}$ such that $x^{\alpha}=x^{\gamma} x_0^{d-\vert\gamma\vert}$, 
where $x^{\gamma}$ belongs to $G(J)$. More precisely, we need to count the number of such variables which do not 
correspond to a pivot in a Gauss reduction of the generators of $L(J)$ (see also Theorem 4.3 of \cite{LR}). 

First we note that in every $S$-polynomial which involves such an $F_{\alpha}$, the polynomial $F_{\alpha}$ 
itself is multiplied by a term in which at least a variable $x_h$ appears, with $h>0$ (otherwise the other 
polynomial involved in the $S$-polynomial should have $x_0^d$ as initial term).
It is enough to investigate the terms $x^{\beta} x_1$, where $x^{\beta}$ belongs to $\mathcal B$, because, 
if $x^{\beta} x_1$ belongs to $J_{d+1}$, then $x^{\beta} x_h$ belongs to $J_{d+1}$ for any $h>0$. Since $J$ is a 
hilb-segment ideal, every term $x^{\beta}$ of $\mathcal B$ is less than $x^{\alpha}$. By the definition of 
$\mathcal B$, every term $x^{\beta}$ of $\mathcal B$ is always involved in a reduction step so that it does not 
appear in any generator of $L(J)$ (see Criterion 4.6 of \cite{LR}).
The number of such terms is at least $\vert G(J)\vert \cdot \vert \mathcal B \vert$ and we are done.
\end{proof}

\begin{theorem}\label{singular point}
If for the hilb-segment ideal $J$ we have $\vert G(J)\vert \cdot \vert \mathcal B \vert > nd$, then $J$ 
corresponds to a singular point in $\hilbd$.
\end{theorem}

\begin{proof} 
Let $H_{RS}$ be the unique irreducible component of $\hilbd$ containing the lexicographic point \cite{RS}. 
Recall that $H_{RS}$ has dimension equal to $nd$ and that every Borel ideal belongs to $H_{RS}$ \cite{RA}. 
Since $J$ is a hilb-segment ideal w.r.t. $\preceq$, the Groebner stratum $St_h(J,\preceq)$ is an open subset of 
$H_{RS}$ (Corollary 6.7 of \cite{LR}) and hence $\dim \ St_h(J,\preceq)=nd$. The point $J$ is smooth for 
$\hilbd$ if and only if it is smooth for the Groebner stratum $St_h(J,\preceq)$ (see Corollary 4.5 of \cite{LR}). 
Thus, $J$ is smooth if and only if $ed(St_h(J,\preceq))= nd$ (Corollary 4.5 of \cite{LR}). By Lemma \ref{lemma d} 
the thesis is proved.
\end{proof}

\begin{example}\label{d=7,8}\rm
Let $I$ be the generic initial ideal of $7$ general points in $\mathbb P^3$ w.r.t. revlex order, i.e. the 
(saturated) hilb-segment ideal with Hilbert polynomial $p(z)=7$. We obtain 
$G(I)=\{x_3^2,x_2 x_3,x_2^2,x_1^2 x_3,x_1^2 x_2,x_1^3\}$ and $\mathcal B=\{x_0^5x_1^2,x_0^5x_1x_2,x_0^5x_2^2\}$. 
Thus $\vert G(I)\vert \cdot \vert \mathcal B\vert =6\cdot 3=18<nd=21$. But, as it is shown in \cite{LR}, 
we can compute directly the Gr\"obner stratum of $I_{\geq 7}$ showing that its embedding dimension is $27>nd=21$. 
Actually, in \cite{LR} the authors construct the stratum of $I_{\geq 3}$ which is isomorphic to the stratum of 
$I_{\geq 7}$, obtaining a big improvement of the computation. 

For $8$ points in $\mathbb P^3$, we have $\mathcal N(I)_8=\{x_0^8,x_0^7x_1,x_0^7x_2,x_0^7x_3,
x_0^6x_1^2,x_0^6x_1x_2,x_0^6x_1x_3,$ $x_0^6x_2^3\}$, so that  $\mathcal B=\{x_0^6x_1^2,x_0^6x_1x_2,x_0^6x_1x_3,x_0^6x_2^3\}$ 
and $\vert \mathcal B\vert=4$. Since $\vert G(I)\vert = 7$, we get $\vert G(I)\vert \cdot \vert \mathcal B\vert = 7\cdot 4 = 28>3\cdot 8=24$.
\end{example}

\begin{theorem}\label{revlexpoint}
For every $d>n\geq 3$, the hilb-segment ideal $J$ w.r.t the revlex order corresponds to a singular point in $\hilbd$. 
\end{theorem}

\begin{proof}
In Remark \ref{numero generatori} we observed that $J$ must  have maximal Hilbert function, so that the regularity 
$\rho_H$ of its Hilbert function is the integer such that $\binom{\rho_H-1+n}{n}< d \leq \binom{\rho_H+n}{n}$. 
Moreover, if $d=\binom{\rho_H+n}{n}$ then $\vert G(J)\vert= \binom{\rho_h+n}{n-1}$, otherwise 
$\vert G(J)\vert \geq \binom{\rho_H+n-1}{n-1}$.

If $d=n+1$ then $\rho_H=1$ and $J=(x_1,\ldots,x_n)^2$, so that $\vert G(J)\vert=\binom{2+n-1}{n-1}=\binom{n+1}{2}$. 
Moreover, $\mathcal B$ consists of the terms of type $x_0^{d-1} x_i$ with $i>0$, thus $\vert \mathcal B\vert=n$ 
and the statement is true because $\binom{n+1}{2}\cdot n > n(n+1)$ for every $n\geq 3$.

If $d\geq n+2$ then $\rho_H\geq 2$. 
  
If $d=\binom{\rho_H+n}{n}$, we show that $\vert \mathcal B\vert>\rho_H+1$. If we multiply every term of degree 
$\rho_H$ in the variables $x_1,\ldots,x_n$ by $x_0^{d-\rho_H}$, we obtain terms of degree $d$ that multiplied by 
$x_1$ give $\binom{\rho_H+n-1}{n-1}$ terms which belong to $\mathcal B$. Thus $\vert \mathcal B \vert \geq 
\binom{\rho_H+n-1}{n-1}>\rho_H+1$ and $\vert G(J)\vert \cdot \vert \mathcal B\vert > \frac{dn}{\rho_H+1}\cdot (\rho_H+1)=dn$. 

If $d<\binom{\rho_H+n}{n}$ and $\rho_H\geq 3$, we show that $\vert \mathcal B\vert \geq \rho_H+n$.
Let $x^{\beta}$ any of the $\binom{\rho_H+n-2}{n-1}$ terms of degree $\rho_H-1$ in the variables $x_1,\ldots,x_n$. 
Thus, if $x^{\beta}x_1$ belongs to $J$, then $x^{\beta}x_0^{d-\rho_H+1}$ belongs to $\mathcal B$; otherwise, 
if $x^{\beta}x_1$ does not belong to $I$, then $x^{\beta}x_0^{d-\rho_H}x_1$ belongs to $\mathcal B$. Anyway, 
the term $x^{\beta}x_1^2$ belongs to $J$ because it is not divided by $x_0$ and has degree $\rho_H+1$ and the 
terms of $\mathcal N(I)_{\rho_H+1}$ are all divided by $x_0$. Such terms are all distinct, so that 
$\vert \mathcal B\vert \geq \binom{\rho_H+n-2}{n-1}$. Now it is easy to check that $\binom{\rho_H+n-2}{n-1}\geq \rho_H+n$, 
for every $\rho_H\geq 3$ and $n\geq 3$. Thus, $\vert G(J)\vert \cdot \vert \mathcal B\vert\geq \binom{\rho_H+n-1}{n-1} \cdot (\rho_H+n)> nd$, 
by Remark \ref{numero generatori}(3).

It remains to study the case $\rho_H=2$ in which $\vert G(I)\vert \geq \binom{n+1}{2}$ and $\vert \mathcal B\vert\geq n$, 
because of the arguments above, with $n=\binom{n+1}{n}< d < \binom{2+n}{n}$. If $d<\binom{n+1}{2}$, then 
 we get immediately $\vert G(J)\vert \cdot \vert \mathcal B\vert>nd$. If $\binom{n+1}{2}<d< \binom{n+1}{2}$, all the 
$d$ terms of $\mathcal N(J)_d$ are in $\mathcal B$ except at most the $n+1$ terms divided by $x_0^{d-1}$. Thus, 
in this case $\vert \mathcal B\vert \geq d-(n+1)$, which is $\geq n+2$ except for $n=3$ and $d=7,8$. These last 
two cases have been directly studied in Example \ref{d=7,8}.
\end{proof}

Observing that:
\begin{itemize}
\item[(i)] a segment ideal w.r.t. revlex order gives rise to a singular point in $\hilbd$ and defines a scheme not contained in any hyperplane;
\item[(ii)] a segment ideal w.r.t. lex order gives rise to a smooth point in $\hilbd$ and defines a scheme contained in some hyperplane;
\end{itemize}
one might guess that there is a relationship between the smoothness of a point in $\hilbd$ corresponding to a (saturated) monomial ideal and the presence of linear forms in the ideal. But, the next example (for which we are indebted to G. Floystad) shows that this is not the case.

\begin{example}
(i) Let $I = (x_1^{a_1},\ldots,x_i^{a_i},\ldots,x_n^{a_n})$ be a (saturated monomial) complete intersection ideal defining a $0$-dimensional scheme $\mathbb X$ of degree $d = \prod_i a_i$ in $\mathbb{P}^n$ and let ${\sf z_{I}}$ denote the corresponding point of of $\hilb^{n}_d$. Being $I$ a monomial ideal, ${\sf z_{I}}$ lies in the closure of the lexicographic point component of $\hilbd$  (see for example Corollary 18.30 of \cite{MS}). Using the normal sheaf to $\mathbb X$, we get that the dimension of the tangent space to $\hilbd$ at ${\sf z_{I}}$ is $nd$, coinciding with that of the lexicographic point component. Thus $I$ gives an example of a monomial ideal which does not contain linear forms and corresponds to a smooth point in $\hilbd$. 

(ii) Let $J \subset K[x_1,\ldots,x_n]$ be a saturated monomial ideal giving a singular point ${\sf z_{J}}$ of $\hilb^{n-1}_d$, so that the dimension of the tangent space to $\hilb^{n-1}_d$ in ${\sf z_{J}}$ is $\alpha > (n-1)d$. Taking $\widetilde{J} = ((x_0)+J) \subset K[x_0,\ldots,x_n]$, the dimension of the tangent space to $\hilbd$ in ${\sf z}_{\widetilde{J}}$ is $\alpha + d > (n-1)d + d = nd$, hence ${\sf z}_{\widetilde{J}}$ is singular too.
\end{example}


\bibliographystyle{amsplain}
\providecommand{\bysame}{\leavevmode\hbox to3em{\hrulefill}\thinspace}
\providecommand{\MR}{\relax\ifhmode\unskip\space\fi MR }
\providecommand{\MRhref}[2]{%
  \href{http://www.ams.org/mathscinet-getitem?mr=#1}{#2}
}
\providecommand{\href}[2]{#2}

\end{document}